\documentclass[a4paper,11pt]{amsart}
\usepackage{comment}
\usepackage{longtable}
\usepackage{listings}

\usepackage[T1]{fontenc}
\usepackage{amscd}
\usepackage{amsmath}
\usepackage{amssymb}
\usepackage{hyperref}
\usepackage{epsf,latexsym,graphicx,dsfont}
\usepackage[matrix,arrow,tips,curve]{xy}
\usepackage{color}
\newlength{\defbaselineskip} 
\usepackage{footnote}
\usepackage{color}
\usepackage{todonotes}
\usepackage{MnSymbol}
\usepackage{amsfonts,amssymb,amscd}
\newtheorem{thm}{Theorem}[section]

\newtheorem{lem}[thm]{Lemma}
\newtheorem{prop}[thm]{Proposition}

\newtheorem{prob}[thm]{Problem}
\theoremstyle{definition}
\newtheorem{defi}[thm]{Definition}
\newtheorem{example}[thm]{Example}
\newtheorem{rem}[thm]{Remark}

\usepackage{fullpage}
\usepackage{tikz,tikz-cd}
\usetikzlibrary{matrix,arrows}
\makeatletter
\usetikzlibrary{shapes.misc}
\usetikzlibrary{patterns}
\usetikzlibrary{shapes.geometric}
\numberwithin{equation}{section} \theoremstyle{definition}

\DeclareMathOperator{\Age}{Age}
 \DeclareMathOperator{\Spec}{Spec}
 
\DeclareMathOperator{\Hom}{Hom}

\newcommand\PP{{\mathbb{P}}}
\newcommand\QQ{{\mathbb{Q}}}
\newcommand\CC{{\mathbb{C}}}
\newcommand\ZZ{\mathbb{Z}}
\newcommand\HK{hyper-K\"ahler\ }
\lstdefinelanguage{Macaulay2}{
  keywords={ dummyIGuess, for, in, while, do, from, if, then, else, and, or, new,},
morecomment=[l]{--},
alsoletter={'},
alsoother={_},
}

\lstdefinelanguage{GAP}
{
    keywords={for, in, while, do, od, if, fi, then, else, end, and, or, function,return, break, local},
    sensitive=true,
    morecomment=[l]{\#},
    morestring=[b]",
    alsoletter = {_}
}

\lstdefinelanguage{MAGMA}
{
    keywords={for, in, while, do, if, then, else, end, and, or, function, procedure, return, break, where, elif},
    sensitive=true,
    morecomment=[l]{//},
    morecomment=[s]{/*}{*/},
    morestring=[b]",
    alsoletter = {_},
    commentstyle=\itshape\color{gray},
}

\author[Maria Donten-Bury]{Maria Donten-Bury}
\address{M. Donten-Bury: Instytut Matematyki UW, Banacha 2, 02-097 Warszawa, Poland}
\email{M.Donten@mimuw.edu.pl}

\author[Grzegorz Kapustka]{Grzegorz Kapustka}
\address{G. Kapustka: Department of Mathematics and Informatics, Jagiellonian University, \L ojasiewicza 6, 30-348, Krak\'ow, Poland}
\email{grzegorz.kapustka@uj.edu.pl}

\author[Benedetta Piroddi]{Benedetta Piroddi}
\address{B.Piroddi: Fachrichtung Mathematik, Campus Saarbr\"ucken, Geb\"aude E2 4, Universit\"at des Saarlandes, 66123 Saarbr\"ucken, Germany}
\email{piroddi@math.uni-sb.de}

\author[Tomasz Wawak]{Tomasz Wawak}
\address{T. Wawak: Department of Mathematics and Informatics, Jagiellonian University, \L ojasiewicza 6, 30-348, Krak\'ow, Poland}
\email{tomasz.wawak@uj.edu.pl}

\begin{document}
\title{A symplectic fourfold}

\begin{abstract}
We present a method to construct irreducible symplectic varieties by studying terminalisations of quotient of \HK manifolds by non-natural group actions.
In particular we construct irreducible symplectic varieties of dimension $4$ with $b_2 = 4$ and non-quotient singularities: this 
provides explicit examples of ISVs for which a global Torelli theorem is not known to hold.

\end{abstract}
\maketitle
\section{Introduction}

Hyper-K\"ahler manifolds are one of the three main building blocks of compact Kähler manifolds with vanishing first Chern class, together with Calabi-Yau manifolds and complex tori \cite{Beauville}. Their rich and interesting geometry is counterbalanced by the scarcity of known examples: K3 surfaces in dimension 2, then two deformation classes in even dimensions (K3$^{[n]}$-type and Km$_n$-type) \cite{Beauville} and two sporadic examples in dimension 6 and 10 \cite{OGrady6, OGrady10}.

However, similar decomposition theorems hold in non-smooth settings: in the settings of orbifolds with vanishing first Chern class \cite{Campana} the objects that play the role of hyper-K\"ahler manifolds are irreducible symplectic orbifolds (ISOs); for normal projective varieties with klt singularities and numerically trivial canonical bundle, a finite cover is the product of irreducible symplectic varieties (ISVs), Calabi-Yau and Abelian klt varieties \cite{HP}. In particular, ISVs seem to be the most interesting setting, as they also appear in the minimal model program, as minimal models of manifolds with trivial first Chern class. It is also worth noting  that if an ISO is projective, then it will also be an ISV;
on the other hand, the singularities appearing in ISVs are not always quotient singularities (a more thorough exposition on the relation between the two notions can be found in~\cite{Perego}).\\

The aim of this paper is to construct an irreducible symplectic fourfold $\tilde X$ with $b_2(\tilde X) = 4$ and non-quotient $\QQ$-factorial terminal singularities (thus providing an answer to Question 2 from~\cite[page 6]{BGMM}): our example admits two singularities that occur as a singular point along a non-minimal nilpotent orbit closure of a Lie algebra.
Such an example is interesting because it is not known if the global Torelli theorem holds for it: indeed, the theorem holds in the singular setting for ISVs with $b_2>4$ \cite[Theorem 1.1]{BL}, and for ISOs \cite{Men20}, but $\tilde X$ belongs to neither of these classes.

There are various strategies in the literature for the construction of symplectic orbifolds and varieties: we recall Fujiki's approach \cite{Fujiki}, later generalised in~\cite{Menet33}; the moduli space construction \cite[Theorem 1.10]{PR}; the compactification of Lagrangian fibrations, first deployed in~\cite{MT} and then generalised in~\cite{ASF, Prym, LLX}.

The approach we are going to follow is to consider terminalisations of quotients of \HK manifolds: let $X$ be a \HK manifold, and let $G$ be a finite group acting symplectically on $X$. 
The quotient $X/G$ is naturally endowed with a symplectic form, that extends to any terminalisation of its singularities $\tilde X$.
Examples of ISVs obtained with this construction for prime order groups can be found in Menet's works (see for instance~\cite{Menet2, MenetCyclic, KapMen}). In~\cite{BGMM}, the authors classify finite symplectic quotients of K3$^{[n]}$-type and Km$_n$-type manifolds, under the assumption that the action of $G$ is \emph{natural}: this means that it deforms to the action induced on the Hilbert scheme of $n$ points $S^{[n]}$ of a K3 surface $S$, by a symplectic action of $G$ on $S$ itself (respectively, deforms to an action on the generalized Kummer variety Km$_n(A)$ coming from the underlying abelian surface).
An alternative construction of K3$^{[2]}$-type manifolds is as Fano varieties of lines of a cubic fourfold, and some non-natural symplectic actions can be induced on K3$^{[2]}$-type manifolds through this construction, starting from actions on cubic fourfolds: a classification of the symplectic varieties that can be obtained terminalising these quotients can be found in~\cite{Mazzon}. 

Our focus in this paper will be on a third construction of K3$^{[2]}$-type manifolds, as double EPW sextics; moreover, we are going to restrict to maximal symplectic group actions: we remark that such actions are never natural, and that if a maximal group $G$ acts on both a double EPW sextic and a Fano variety of lines on a cubic fourfold, usually the two actions will be different in cohomology (see Appendix \ref{Ap:groups} for more details and exceptions).\\

A classification of all finite symplectic actions on K3$^{[2]}$-type manifolds can be found in~\cite{HM}.
Our main construction concerns the action of the group $L_3(4)$, the second largest in the list.
Note that there are two one parameter families of \HK fourfolds admitting an action of this group, but the action is unique on a given projective example. Our main result is the following.

\begin{thm}\label{main}
Let $X$ be a general $\text{K}3^{[2]}$-type manifold admitting a symplectic action with invariant lattice 
 \begin{equation}\label{eq:TX}
     \begin{bmatrix}
        2 &0 &0 \\ 0 &10 &4\\ 0 &4 &10 
    \end{bmatrix}.
 \end{equation}
The $\QQ$-factorial terminalisation $\tilde{X}$ of the fourfold ${X/L_3(4)}$ is simply connected, with simply connected smooth locus, and has $b_2(\tilde{X})=4$. Its singularities are nine of type $\frac{1}{7}(1,2,-1,-2)$ twelve of type $\frac{1}{5}(1,2,-1,-2)$ six of type $\frac{1}{3}(1,-1,1,-1)$ sixteen of type $\frac{1}{2}(1,1,1,1)$ and two singular points analytically isomorphic to $(\mathcal{Z}(5),x_5)$.
\end{thm}
The singularity $(\mathcal{Z}(5),x_5)$ appearing in the statement is an isolated $\QQ$-factorial singularity: a description of it can be found in~\cite[Section 1]{BBFJLS}. It was observed in~\cite[Theorem 1.3]{BBFJLS} that this singularity is not locally analytically isomorphic to a singularity of a minimal orbit closure for any simple Lie algebra; moreover, since it has trivial local fundamental group, it is not a quotient singularity.

This is a new example, as far as we can tell from the literature: indeed, we remark that for many known examples of ISVs, the second Betti number and the singularities are not known (a partial list can be obtained from the classifications in~\cite{BGMM, Menet33, Mazzon}).

\begin{rem}\label{rem:dimension_of_moduli_tildeX}

The moduli space of marked ISVs deformation equivalent to our new example has dimension $2$. We start from a rigid polarised K3$^{[2]}$-type manifold $X$: the $\QQ$-factorial terminalisation $\tilde X$ of the quotient $X/L_3(4)$ (that exists by~\cite[Corollary 1.4.3]{BCHM2010}) 
 is polarised, its polarisation coming from the $L_3(4)$-invariant polarisation of $X$; moreover, the exceptional divisor defines a new algebraic class. Therefore, $\tilde X$ has two algebraic classes, while its general deformation (a general point in the corresponding period domain) has no algebraic class. Note also that (losing the polarisation of $X$) there exists a one-dimensional family of marked K3$^{[2]}$-type manifolds admitting a symplectic action of $L_3(4)$. Since by~\cite[Proposition 3.9]{BayerPerry} the symplectic action can be performed in families (after possibly a base change), taking quotients and terminalising we obtain a one-parameter family of singular ISVs: this gives a divisor in the period domain corresponding to examples where the exceptional divisor of the terminalisation is algebraic.
\end{rem}

The proof of Theorem \ref{main} is given in Section \ref{sec:proof_of_main}.
The idea is to consider a special symmetric K3$^{[2]}$-type fourfold $X$ with a symplectic $L_3(4)$ action, having additionally a polarisation of Beauville-Bogomolov degree $2$ (it exists by~\cite[Table 1]{Wawak}). This permits us to study in Section \ref{sec:singularities} the singularities of $X/L_3(4)$. More precisely, such a polarisation defines a $2:1$ map $X\to Y\subset \PP^5$, where $Y$ is invariant for a linear action of $L_3(4)$ on $\PP^5$: in Section \ref{sec:fixed_points} we use the latter action to deduce the points with non-trivial stabiliser for the action of $L_3(4)$ on $X$. In Section \ref{sec:terminalisations}, we study their terminalisations: in particular, the two singularities of type $(\mathcal{Z}(5),x_5)$ come from points fixed by $D_{10}$.\\

In Appendix \ref{Ap:K3} we describe the isomorphism class of the K3 surfaces fixed on $X$ by any of the involutions in $L_3(4)$.

In Appendix \ref{Ap:groups} we discuss other maximal groups ($A_7, L_2(11), Z_2\times A_2(7), \mathbb{Z}_2^4:S_5$ and $M_{10}$) 
that can be used in place of $L_3(4)$ to provide new examples of ISVs. In particular, we find in Theorem \ref{main1} that the terminalisations of the quotient manifolds with respect to the groups $A_7, L_2(11), M_{10}$ also have $b_2=4$, and the first two of them surely have non-quotient singularities.

In Appendix \ref{Ap:codes} we provide the code to construct the EPW sextic $Y\subset\PP^5$ invariant for the action of $L_3(4)$.

In Appendix \ref{Ap:fix} we provide tables for the invariant subspaces in $\PP^5$ of subgroups of the maximal groups studied in this paper, and their intersection with invariant EPW sextics $Y\subset\PP^5$.

\subsection*{Acknowledgements}
We would like to thank M. Mauri, G. Mongardi, A. Sarti and C. Tschanz, for discussions.
GK is supported by the project Narodowe Centrum Nauki 2024/53/B/ST1/00161.
BP is supported by the project Narodowe Centrum Nauki 2024/55/B/ST1/02409, the project Narodowe Centrum Nauki 2018/30/E/ST1/00530 and was supported by the program Excellence Initiative at the Jagiellonian University in Krakow (ID.UJ).
TW is supported by the project Narodowe Centrum Nauki 2024/53/B/ST1/01413 and the project Narodowe Centrum Nauki 2018/30/E/ST1/00530.

\section{General results}
\subsection{Double EPW sextics} 
Let us first recall the classical double EPW construction.

Let $W$ be a $6$-dimensional complex vector space. Let \[\textstyle \operatorname{LG}(10,\bigwedge^3 W ) \subset \operatorname{G}(10, \bigwedge^3 W )\]
be the Lagrangian Grassmannian corresponding to the symplectic structure on $\bigwedge^3 W$ given by wedge product and by fixing a generator of $\bigwedge^6 W$. 
 Set
\begin{equation}\label{Fv}\textstyle
F_v=v\wedge \bigwedge^2 W \subset  \bigwedge^3 W \text{ for } v\in W;
\end{equation}
notice that all these spaces are Lagrangian subspaces of $\bigwedge^3 W$. Then we define
\begin{equation}\textstyle
	Y_A:=\{[v]\in \mathbb P(W)\mid\ \dim (F_v\cap A)\geq 1 \},
 \end{equation}
the Eisenbud-Popescu-Walter (EPW) sextic associated to a Lagrangian space $A \subset \bigwedge^3 W$.
When $A$ is general, canonical double covers of $Y_A$ were constructed: they are \HK manifolds of $\text{K}3^{[2]}$-type with a polarisation of Beauville-Bogomolov degree $2$ \cite{Ogrady-EPW}.

If $\PP(A)$ is the span of the second Veronese $v_2(\PP^3)\subset G(3,6)\subset \PP(\bigwedge^3 W)$, defined by the planes contained in a quadric in $\PP^5$, the corresponding \HK manifold is the Hilbert scheme of two points on a quartic surface. The map to $\PP^5$ is given by associating to two points on the K3 surface a line spanned by them (thus, a point in $G(2,4)\subset \PP^5$).

In order to construct a double EPW sextic with a symplectic group action, we need to find a Lagrangian space $A\subset \bigwedge^3 W$ having an appropriate stabilizer: depending on the choice of the group, the resulting $A$ might be special (see Appendix \ref{Ap:groups}).

\subsection{A double EPW sextic with an $L_3(4)$ symplectic action}
It follows from~\cite[Table 9]{HM} that one can define two different symplectic actions of the group
$G=L_3(4)$ on a $\text{K}3^{[2]}$-type manifold $X$. The co-invariant lattice $L_G$ for the action of $L_3(4)$ on $H^2(X,\ZZ)$ is uniquely defined and is of maximal rank 20; the invariant lattice $L^G$ can be either of the following positive definite lattices:
 \begin{equation}\label{eq:invariant}
     \begin{bmatrix}
        2 &0 &0 \\ 0 &10 &4\\ 0 &4 &10 
    \end{bmatrix},\ \ \ \ \ 
     \begin{bmatrix}
        4 &2 &0 \\ 2 &4 &0\\ 0 &0 &14 
    \end{bmatrix}.
 \end{equation}

The data $(L_G,L^G)$ defines a unique action of $L_3(4)$ on $H^2(X,\ZZ)$, so by the Torelli theorem for \HK manifolds one can deduce that, for a general element $X$ in one of the two families, the action of $L_3(4)$ on $X$ is unique (depending only on the family).

Notice that, if we now consider \emph{polarised} elements, once we fix the degree of the polarisation the moduli space becomes discrete, if not empty. In particular, it follows from~\cite[Table 1]{Wawak} that in the first family there is an element $(X,H)$, where $H$ is an $L_3(4)$ invariant polarisation of Beauville-Bogomolov degree $2$, and the transcendental lattice of $X$ is
  \[T(X)=\begin{bmatrix}
     10 &4\\ 4 &10 
    \end{bmatrix}.\]
  \begin{lem} The polarisation $H$ gives a $2:1$ map $\pi\colon X\to Y\subset\PP^5$, where $Y$ is an EPW sextic that is also invariant with respect to the $L_3(4)$ action.
\end{lem}
\begin{proof}
Since the polarisation $H$ is invariant, $L_3(4)$ acts on the space of its global sections. 
Similarly to~\cite{BilliWawak}, we find in Appendix \ref{Ap:codes} a symmetric Lagrangian space $A$ in $\bigwedge^3 \CC^6$ for the action of $L_3(4)$; as it is shown in~\cite{BMW} 
that  $\PP(A)$ does not intersect the Grassmanian $G(3,6)$ and is outside the locus $\Delta$, it follows that there exist a double EPW sextic with a symplectic action of $L_3(4)$. This has to be our example, as it is unique with this property.
\end{proof}
\begin{example}\label{modelofY}
In order to produce a suitable model of the EPW sextic $Y\subset\PP^5$, in our computations we consider the action of $L_3(4)$ on $\mathbb P^5$ generated by the matrices 
\[\left[\begin{array}{c c c c c c} 
-1 &0 &0 &0 &0 &0\\
0 &0 &1 &0 &0 &0\\
0 &1 &0 &0 &0 &0\\
0 &0 &0 &0 &1 &0\\
0 &0 &0 &1 &0 &0\\
0 &0 &0 &0 &0 &-1\end{array}\right],\quad 
\left[\begin{array}{c c c c c c}
0 &1 &0 &0 &0 &0\\
0 &-1 &0 &0 &1 &\zeta_3+1\\
0 &0 &-1 &0 &0 &0\\
0 &0 &-(\zeta_3+1) &0 &0 &1\\
-1 &-1 &0 &\zeta_3+1 &0 &0\\
0 &0 &0 &-1 &0 &0
\end{array}\right],\]
and take $Y\subset \mathbb P^5$ as the general EPW sextic invariant for this action. See also Appendix \ref{Ap:codes} for the Macaulay2 code describing $Y$.
\end{example}

\subsection{Group actions}
Assume that a group $G$ acts on an EPW sextic $Y$: then, on the \HK double cover $\pi\colon X\to Y$ we have an action of the group $\langle G,\iota\rangle$, where $\iota$ is the covering involution and acts non-symplectically on $X$. If $|G|$ is even and $X$ admits a symplectic action of $G$, it also admits a non-symplectic action of $G$ given by $\iota\circ G$; if instead $|G|$ is odd, the lift of the action to $X$ is unique. Therefore, we have the following:

\begin{rem}\label{rem:X_and_Y}\textit{Lifting group actions from $Y$ to $X$.}
Assume that $G$ fixes a point $y\in Y$ outside of the branch locus.
If $|G|$ is even, and $G$ acts symplectically on $X$, it might fix each one of the two points $x_1,x_2\in\pi^{-1}(y)$, or exchange them: in the latter case, the stabiliser of $x_i$ will be an index two subgroup $H\subset G$.
If $|G|$ is odd, then $x_1,x_2\in\pi^{-1}(y)$ will be fixed by $G$.
\end{rem}

In particular, given the Remark above, we need to understand how involutions on $Y$ lift to symplectic involutions on $X$.
\begin{rem}\label{fixed_locus_involution}
    \textit{The fixed locus of an involution on a double EPW sextic} \cite[Section 4]{CGKK}.
Assume the double EPW sextic $\pi\colon X\to Y$ admits a symplectic involution, and let $\iota$ be the induced involution on $\PP^5$ ($\iota(Y)=Y$): then the invariant subspaces for the action of $\iota$ are a copy of $\PP^1$ and a copy of $\PP^3$. It holds that $\PP^3\cap Y$ is the union $Q\cup K$ of a quadric surface and a quartic surface with 16 nodes (a Kummer surface). Recall that the singular locus $Y^{sing}$ is the branch locus of the double cover $\pi\colon X\to Y$; the intersection of $Y^{sing}$ with the invariant subspaces of $\iota$ consists of $Q\cap K$ and the 16 nodes of $K$ (so it is entirely contained in the invariant $\PP^3$). \\
A symplectic involution on a K3$^{[2]}$-type manifold always fixes a K3 surface and 28 isolated points \cite{Mongardi}; in our case, the K3 surface is the double cover $F$ of $Q$ (see also Appendix \ref{Ap:K3}), 12 of the isolated points are in $\pi^{-1}(\PP^1\cap Y)$ and the remaining 16 are preimages of the nodes of $K$.
\end{rem}

We are going to also use the following observations in our study of the orbits of points with non-trivial stabiliser.
\begin{lem}\label{lem:orbit_of_points}
Consider the action of a finite group $G$ on a set $\mathcal S$, and let $H$ be a subgroup of $G$: then it holds $Fix(G)\subset Fix(H)$. Moreover, denote $[H]$ the conjugacy class of $H$ in $G$, and assume $[H]$ has $|G|/|H|$ elements. Then, if a point $p$ is fixed by $H$, its stabiliser is exactly $H$, unless $p$ is also fixed by another subgroup in $[H]$.
\end{lem}

\section{Points of $\texorpdfstring{X}{X}$ with nontrivial stabiliser under the action of $\texorpdfstring{L_3(4)}{L3(4)}$}\label{sec:fixed_points}

We begin by studying the locus of points of a K3$^{[2]}$-type manifold that have nontrivial stabiliser under the symplectic action of $L_3(4)$.
Recall that, since this group does not act symplectically on a K3 surface, its action on a K3$^{[2]}$-type manifold cannot be natural; moreover, the invariant lattice for this action (the first in \eqref{eq:invariant}) does not contain a class of square 6 and divisibility 2. Therefore, our group is not included in the lists of groups studied in~\cite{BGMM, Mazzon}).
However, as already discussed in the previous section, there exists a projective model of a K3$^{[2]}$-type manifold with a symplectic action of $L_3(4)$ as double EPW sextic, given by an $L_3(4)$-invariant polarization $H$ of Beauville-Bogomolov square $2$. Denote $\pi\colon X\to Y\subset \PP^5$ the projective map given by global sections of $H$: then, we can assume that $Y$, and the action of $L_3(4)$ on it, are as described in Example \ref{modelofY} (and implemented with the codes in Appendix \ref{Ap:codes}). We are going to rely on this projective model of $Y$ for many of our computations.
\begin{rem}\label{rem:preimages_are_fixed_by_iota}
    Notice that we are able to intersect the invariant subspaces for any element $g\in L_3(4)$ with the EPW sextic $Y$. In particular, for each involution $\iota_k$ we can find explicitly the surfaces $Q_k$ and $K_k$ and the 6 points given by the intersection of $Y$ with the invariant subspace $\PP^1_k$, as described in Remark \ref{fixed_locus_involution}, allowing us to also infer whether two points $x_1,x_2$ on the double cover $X$ of $Y$ are exchanged or not, assuming they lie on the fiber of a regular point $y\in Y$ fixed by an involution $\iota_k$: indeed, in this case $x_1,x_2$ would be fixed by the lift of $\iota_k$ only if $y\in Q_k\cup (Y\cap\PP^1_k)$, as the preimages of regular points of $K_k$ are not fixed.
\end{rem}

The elements of $L_3(4)$ have order 2, 3, 4, 5 and 7. The 315 involutions $\iota_k$ are all conjugate, and they generate the whole group. All elements of order 3 are also conjugated. There are two power-equivalent conjugacy classes of elements of order 5 and 7; recall however, that if $g$ has odd prime order, then its non-trivial powers have the same fixed locus (and, in our case, all these elements fix isolated points). 
Elements of order 4 are split in three non power-equivalent conjugacy classes, and understanding their interaction is crucial in determining the singularities of $X/L_3(4)$.

\begin{rem}\label{rem:fixedK3}\textit{The surface $S$.}
The singular locus of $X/L_3(4)$ contains a surface $S$, that is the image in $X/L_3(4)$ of the K3 surfaces $F_k$ fixed by each of the conjugate involutions $\iota_k$ (see Remark \ref{fixed_locus_involution}). The generic point of $S$ is an ordinary double point for $X/L_3(4)$.
\end{rem}

\subsection{Fixed points of subgroups of $\texorpdfstring{L_3(4)}{L3(4)}$ generated by two involutions}\label{sec: subgroups_stabilising_points}
The surface $S$ introduced in Remark \ref{rem:fixedK3} contains some points that are not ordinary double points for the quotient manifold: by Lemma \ref{lem:orbit_of_points} these points lie in the intersection between two K3 surfaces $F_j, F_k$,
and to study these singularities, we have to look at the subgroups of $L_3(4)$ generated by two involutions.

\begin{lem}\label{groups_gen_by_involutions}
Two non-commuting involutions $\iota_j,\iota_k\in L_3(4)$ generate one of the following groups: $\ZZ_2\times \ZZ_2$, the symmetric group $S_3$, and the dihedral groups $D_8, D_{10}$ ($D_n$ has cardinality $n$). Moreover, every element of order 3, 4, 5 is contained in a copy of $S_3, D_8, D_{10}$ respectively, and every copy of $\ZZ_2\times \ZZ_2$ is contained in a copy of $D_8$.
\end{lem}

\begin{rem}\label{rem:natural_action}
\textit{Naturality.} The subgroups introduced in Lemma \ref{groups_gen_by_involutions} indeed act naturally on $X$. For $D_8$ and $D_{10}$ this follows directly from~\cite[Table 12]{HM}: these groups can act only naturally on a K3$^{[2]}$-type manifold. Symplectic automorphisms of order 3 admit a second possible (non-natural) action, that fixes an abelian surface: however, the elements of order 3 of $L_3(4)$ fix only isolated points on $Y$, so we can conclude that the action of $S_3$ on $X$ is natural too (see Appendix \ref{Ap:fix}).
\end{rem}

We will first consider, in Sections \ref{sec:S3}, \ref{sec:D8_on_Y}, \ref{sec:D10} the action of \emph{one} representative element $H$ in each of the conjugacy classes of the nonabelian subgroups of $L_3(4)$ presented in Lemma \ref{groups_gen_by_involutions}: for each $H$, we are going to compute the locus of points in $X$ with nontrivial stabiliser under its action (thus, since any copy of $\ZZ_2\times \ZZ_2$ is contained in some copy of $D_8$, the points the former fixes will be found while studying the latter). This is done by comparing the natural action of $H$ on a K3$^{[2]}$-type manifold, with the action the projective model $Y\subset \mathbb{P}^5$ given in Example \ref{modelofY} and implemented in Appendix \ref{Ap:codes}. In particular, using our computer model of $Y$, we can also see the interplay between the different subgroups of $L_3(4)$ isomorphic to $H$, that may act on the same set of points; moreover, since any subgroup of $L_3(4)$ (other than $\ZZ_2$) fixes isolated points $p\in Y$, for each such $p$ we can directly compute its stabiliser, that might be bigger than $H$ itself (see Section \ref{sec:points_bigger_stabiliser}). On the other hand, understanding the natural actions of the groups presented in Lemma \ref{groups_gen_by_involutions} is a fundamental step to deduce the points with nontrivial stabiliser on the double cover $X\rightarrow Y$, according to Remarks \ref{rem:X_and_Y} and \ref{rem:preimages_are_fixed_by_iota}.
We will see later, in Section \ref{sec:singularities}, that the action of $L_3(4)$ may further identify some of the points with nontrivial stabiliser found in this section.

\subsubsection{The action of $S_3$}\label{sec:S3}
The group $S_3$ contains 3 involutions $\iota_1,\iota_2,\iota_3$, and the product of any two of them gives an element $g_3$ of order 3 or its square. Recall from~\cite[Theorem 7.2.7]{MongardiThesis} that a natural symplectic automorphism of order 3 on a K3$^{[2]}$-type manifold fixes 27 isolated points. The invariant subspaces for the action of $g_3$ on $\PP^5$ are three copies of $\PP^1$, each intersecting $Y$ in one singular point $p$ and four regular points $q_1,\dots,q_4$.  
Moreover, each $\PP^1$ intersects 3 different quadric surfaces fixed by involutions (as described in Remark \ref{fixed_locus_involution}), such that $\bigcap_{i=1,2,3}Q_i=\bigcap_{i=1,2,3}Q_i\cap\PP^1$ consists of the singular point $p$
and one of the regular points, say $q_1$. Indeed, the element $g_3$ is the center of three different copies of $S_3$. Therefore, each copy of $S_3$ fixes 3 points on $X$ (the intersection of the three K3 surfaces  $F_i$ fixed by its involutions), leaving 18 isolated points fixed by the lift of $g_3$ on $X$. 

\subsubsection{The action of $D_8$}\label{sec:D8_on_Y}
The group $D_8$ contains two elements of order 4; $\alpha$ and $\alpha^3$, and 5 involutions $\alpha^2, \iota_1,\dots,\iota_4$ with the property that $\iota_1\iota_3=\iota_2\iota_4=\alpha^2$, while any other product $\iota_i\iota_j$ has order 4.\\

\textit{On a K3 surface.} 
The group $D_8$ acts uniquely on a K3 surface $\Sigma$ \cite{Hashimoto}: to describe its action, we compare the singularities of the quotient $\Sigma/D_8$ (given in~\cite[Theorem 3]{Xiao}) to that of the quotient of $\Sigma$ by its normal subgroups $\ZZ_2, \ZZ_4$. 

We find that each involution in $D_8$ fixes exactly 8 points on $\Sigma$, and these points are all distinct, giving 40 points of order 2 on $\Sigma$. Of the 8 points fixed by $\alpha^2$, 4 are also fixed by $\alpha$. Denote $Fix(\iota_k)=\{r_1^k,\dots, r_8^k\},\ Fix(\alpha^2)=\{p_1,\dots,p_4,q_1,\dots,q_4\}$, and assume $\alpha$ fixes $\{p_1,\dots,p_4\}$, while $\alpha(q_1)=q_2,\ \alpha(q_3)=q_4$. Then, since $\iota_1\alpha^2=\iota_3$, $\alpha^2$ acts as $\iota_1$ on $Fix(\iota_3)$, and as $\iota_3$ on $Fix(\iota_1)$, while $\iota_1,\iota_3$ act the same on $Fix(\alpha^2)$; the same holds for $\iota_2, \iota_4$ instead. Moreover, since $\iota_1\iota_2=\alpha$, we can assume it holds $\iota_1(p_1)=\iota_2(p_1)=p_2,\iota_1(p_3)=\iota_2(p_3)=p_4$, but $\iota_1(q_1)=q_4,\iota_1(q_2)=q_3$ while $\iota_2(q_1)=q_3,\iota_2(q_2)=q_4$.\\

\textit{On a K3$^{[2]}$-type manifold.} The action of $D_8$ on a K3$^{[2]}$-type manifold is always natural \cite{HM}, so (by~\cite[Proposition 2.6]{Fujiki}, and the fact that K3 surfaces and points deform to K3 surfaces and points) 
it will have the same fixed locus as the induced action of $D_8$ on $\Sigma^{[2]}$. Denote $\Delta$ the divisor of $\Sigma^{[2]}$ parametrising non-reduced subschemes, $F_k$ the K3 surface fixed by $\iota_k$ (it is the closure of the set of points $\{[x,\iota_k(x)]\}$) and $F_\alpha$ the K3 surface fixed by $\alpha^2$.
Notice that $F_\alpha$ intersects $\Delta$ in 4 lines over the points $[p_i,p_i]$, on which the action of $\alpha$ fixes 2 points each (8 points on $F_\alpha$ in total), and 4 lines $[q_i,q_i]$, exchanged by the action of $\alpha$. The 28 points fixed by $\alpha^2$ are of one of the following types: $[p_i,q_j]$ for any $i,j$ (16 points), $[p_i,p_j],[q_i,q_j]$ for $i\neq j$ (12 points); of these points, $\alpha$ fixes all $[p_i,p_j]$, and $[q_1,q_2],[q_3,q_4]$. The 28 points fixed by $\iota_k$ are $\{[r_i^k,r_j^k]\}$ for $i\neq j$.

The intersection of the surfaces fixed by any two commuting involutions consists of 4 of the isolated points fixed by their product. In particular, $F_1\cap F_3=\{[p_1,p_2],[p_3,p_4],$ $[q_1,q_4],[q_2,q_3]\}$, while $F_2\cap F_4=\{[p_1,p_2],[p_3,p_4],[q_1,q_3],[q_2,q_4]\}$; taking two noncommuting involutions we obtain the only two points fixed by the whole group $D_8$,
\[F_1\cap F_2=\bigcap_{k=1,\dots,4}S_k=\{[p_1,p_2],[p_3,p_4]\}.\]
Notice last that the six points $[p_1,p_3],[p_2,p_4],[p_1,p_3],[q_1,q_2],[q_3,q_4]$ are isolated points fixed by $\alpha$, but by none of the involutions $\iota_k$.\\

\textit{On the double EPW sextic.} There are three conjugacy classes of elements of order 4 in $L_3(4)$, and every element of order 4 is contained in a copy of $D_8$. Denote $\alpha,\beta,\gamma$ a choice of representatives for the three classes, and consider their action on $\PP^5$: then, $\alpha$ has three copies of $\PP^1$ as invariant subspaces, while the invariant subspaces of $\beta$ or $\gamma$ are two points and two copies of $\PP^1$ (see also the table in Appendix \ref{Ap:fix}). We are going to study separately the action of $D_8$ on $Y$ when the normal element of order 4 is $\alpha$ (\textbf{case 1}), or $\beta,\gamma$ (both acting as in \textbf{case 2}): notice that, while cases 1 and 2 are different on $Y$, they both lift on $X$ to a deformation of the natural action described above.
\begin{itemize}
    \item \textbf{case 1:} Two of the invariant $\PP^1$ subspaces for the action of $\alpha$ on $\PP^5$ are contained in the quadric surface $Q_{\alpha}\subset Y$ fixed by $\alpha^2$, and each contains 4 singular points. The remaining $\PP^1$ invariant subspace is the intersection of four copies of $\PP^3$, each fixed by an involution $\iota_k$ in the copy of $D_8$ where $\alpha$ is normal; $\PP^1\cap Y$ consists of 6 regular points $p_1,\dots, p_6$, of which $p_1=\bigcap_k Q_k$ is the common intersection of the quadric surfaces fixed by each $\iota_k$, $p_2=Q_1\cap Q_3$, $p_3=Q_2\cap Q_4$. 
    \begin{rem}\label{liftedaction}\textit{Lifting the action.} Notice that $p_2,p_3$ are fixed by $D_8$ on $Y$ but, comparing this action with the natural action of $D_8$ on a K3$^{[2]}$-type manifold, we see that the two points on each of their fibers have to be exchanged on $X$ (so their stabiliser is $(\ZZ_2)^2$). Similarly, the points on the fiber of $p_1$ are fixed by $D_8$, and the points on the fibers of $p_4,p_5,p_6$, (which are regular points in $\PP^3_k$, and are not contained in the respective quadric surfaces) are fixed by $\ZZ_4$.
    \end{rem}
    Consider now the stabilisers of these points for the action of the whole $L_3(4)$. The double cover $F_\alpha\rightarrow Q_{\alpha}$ restricts to a double cover of each $\PP^1$ made of two rational curves intersecting in 4 points, that are the pre-image of the singular points in $Q_{\alpha}$; the stabiliser of these 8 points in $F_\alpha$ is the lift of $\alpha$ to $X$. On the set $p_1,\dots, p_6$ there are \emph{four} copies of $D_8$ acting, such that each $p_i\in \{p_1,\dots,p_4\}$ is the common intersection of the quadric surfaces fixed by the involutions $\iota_k$ for exactly one of these copies of $D_8$; denoting them $(D_8)_i$ accordingly, we see that $(D_8)_1$ and $(D_8)_4$ have no common involution (other than $\alpha^2$), and similarly $(D_8)_2$ and $(D_8)_3$; at the same time, $(D_8)_2$ and $(D_8)_3$ share two involutions each with each of $(D_8)_1$ and $(D_8)_4$, so that there are only 8 different involutions (plus $\alpha^2$) generating these 4 copies of $D_8$. The stabiliser of $p_i\in \{p_1,\dots,p_4\}$ in $L_3(4)$ is $D_8\times \mathbb Z_2$, but by Remark \ref{liftedaction} we deduce that on $X$ each of the two points on the fiber of $p_i$ is stabilised only by $D_8$.
    
    The stabiliser of the two remaining points $p_5,p_6$ is a group of order 32, and it will be treated separately. \\ 
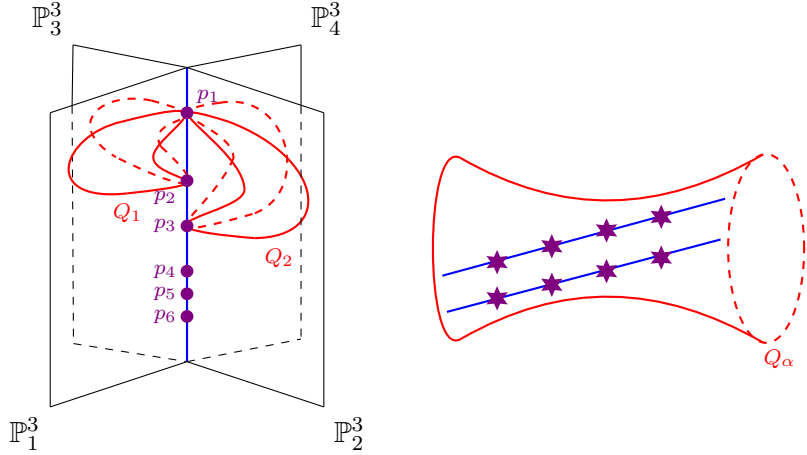
\begin{figure}
\centering
\centering
\begin{tikzpicture}[scale=0.6]
  % Coordinates
  \coordinate (A) at (2.00, 8.00);
  \coordinate (B) at (2.00, 1.50);
  \coordinate (C) at (5.00, 9.00);
  \coordinate (D) at (5.00, 2.50);
  \coordinate (E) at (8.00, 8.00);
  \coordinate (F) at (8.00, 1.50);
  \coordinate (G) at (2.50, 9.50);
  \coordinate (H) at (2.48, 8.15);
  \coordinate (I) at (2.51, 2.90);
  \coordinate (P) at (5.92, 7.80);
  \coordinate (Q) at (6.48, 5.23);
  \coordinate (R) at (8.00, 7.00);
  \coordinate (S) at (8.25, 5.25);
  \coordinate (T) at (6.26, 6.30);
  \coordinate (U) at (6.25, 5.75);
  \coordinate (V) at (3.25, 5.50);
  \coordinate (W) at (3.75, 8.50);
  \coordinate (X) at (6.25, 7.25);
  \coordinate (Y) at (4.22, 8.27);
  \coordinate (Z) at (3.15, 8.50);
  \coordinate (AA) at (2.25, 7.75);
  \coordinate (AB) at (3.44, 7.02);
  \coordinate (AC) at (4.43, 7.25);
  \coordinate (AF) at (6.07, 7.76);
  \coordinate (AG) at (4.00, 8.00);
  \coordinate (AD) at (5.25, 6.50);
  \coordinate (AE) at (5.50, 6.00);
  \coordinate (AH) at (3.56, 7.80);
  \coordinate (AI) at (3.70, 6.22);
  \coordinate (AJ) at (2.25, 7.50);
  \coordinate (AK) at (1.75, 6.00);
  \coordinate (AL) at (4.26, 7.02);
  \coordinate (AM) at (4.25, 6.50);
  \coordinate (AN) at (6.25, 6.50);
  \coordinate (AO) at (6.25, 8.50);
  \coordinate (AP) at (4.25, 7.50);
  \coordinate (AQ) at (5.00, 8.00);
  \coordinate (AR) at (5.00, 6.50);
  \coordinate (AS) at (5.00, 5.50);
  \coordinate (AT) at (5.00, 4.50);
  \coordinate (C_1) at (5.00, 9.00);
  \coordinate (D_1) at (5.00, 2.50);
  \coordinate (J) at (7.50, 9.50);
  \coordinate (K) at (7.50, 8.15);
  \coordinate (L) at (7.50, 3.00);
  \coordinate (M) at (5.72, 8.22);
  \coordinate (N) at (7.25, 8.50);
  \coordinate (O) at (7.75, 6.25);
  \coordinate (AU) at (6.46, 5.77);
  \coordinate (AV) at (4.00, 8.00);
  \coordinate (AW) at (6.00, 7.75);
  \coordinate (AX) at (6.02, 7.00);
  \coordinate (AY) at (5.50, 5.75);
  \coordinate (BA) at (3.75, 5.00);
  \coordinate (BB) at (5.00, 4.00);
  \coordinate (BC) at (5.00, 3.50);

  % Drawing
  \draw (A) -- (B);
  \draw (A) -- (C);
  \draw[blue, thick] (C_1) -- (D_1);
  \draw (B) -- (D);
  \draw (C) -- (E);
  \draw (E) -- (F);
  \draw (D) -- (F);
  \draw (G) -- (C);
  \draw (G) -- (H);
  \draw[dashed] (H) -- (I);
  \draw[dashed] (I) -- (D);
  \draw[red, thick] (P) .. controls (8.00, 7.00) and (8.25, 5.25) .. (Q);
  \draw[red, thick] (T) .. controls (6.25, 5.75) and (3.25, 5.50) .. (Q);
  \draw[red, thick] (P) .. controls (3.75, 8.50) and (6.25, 7.25) .. (T);
  \draw[red, thick, dashed] (Y) .. controls (3.15, 8.50) and (2.25, 7.75) .. (AB);
  \draw[red, thick, dashed] (Y) .. controls (6.07, 7.76) and (4.00, 8.00) .. (AC);
  \draw[red, thick, dashed] (AC) .. controls (5.25, 6.50) and (5.50, 6.00) .. (AB);
  \draw[red, thick] (AH) .. controls (2.25, 7.50) and (1.75, 6.00) .. (AI);
  \draw[red, thick] (AL) .. controls (4.25, 6.50) and (6.25, 6.50) .. (AI);
  \draw[red, thick] (AH) .. controls (6.25, 8.50) and (4.25, 7.50) .. (AL);
  \draw (C) -- (J);
  \draw (J) -- (K);
  \draw[dashed] (K) -- (L);
  \draw[dashed] (D) -- (L);
  \draw[red, thick, dashed] (M) .. controls (7.25, 8.50) and (7.75, 6.25) .. (AU);
  \draw[red, thick, dashed] (M) .. controls (4.00, 8.00) and (6.00, 7.75) .. (AX);
  \draw[red, thick, dashed] (AX) .. controls (5.50, 5.75) and (3.75, 5.00) .. (AU);
  \fill[violet] (AQ) circle (4pt);
  \node[above right, violet] at (AQ) {$\scriptstyle p_1$};
  \fill[violet] (AR) circle (4pt);
  \node[below left, violet] at (AR) {$\scriptstyle p_2$};
  \fill[violet] (AS) circle (4pt);
  \node[left, violet] at (AS) {$\scriptstyle p_3$};
  \fill[violet] (AT) circle (4pt);
  \node[left, violet] at (AT) {$\scriptstyle p_4$};
  \fill[violet] (BB) circle (4pt);
  \node[left, violet] at (BB) {$\scriptstyle p_5$};
  \fill[violet] (BC) circle (4pt);
  \node[left, violet] at (BC) {$\scriptstyle p_6$};
  \node[below left, black] at (B) {$\mathbb P^3_1$};
  \node[below right, black] at (F) {$\mathbb P^3_2$};
  \node[above left, black] at (G) {$\mathbb P^3_3$};
  \node[above right, black] at (J) {$\mathbb P^3_4$};
  \node[below right, red] at (Q) {$\scriptstyle Q_2$};
  \node[below, red] at (AI) {$\scriptstyle Q_1$};

%hyperboloid (quadric Q_\alpha)
  \node[draw, thick, red, dashed, ellipse, minimum width=1cm, minimum height=2.5cm, outer sep=0] (ell) at (17.7,5) {};
\draw[red, thick] (ell.92) 
    to[out=210,in=-30] (11,7) 
    to[out=150,in=210, looseness=.6] (11,3)
    to[out=30,in=150] (ell.-92);
\node[below, red] at (18,3) {$\scriptstyle Q_\alpha$};

\draw[blue, thick] (10.6,4.4) to (16.8,6.1);
\draw[blue, thick] (10.7,3.6) to (16.7,5.2);

\node [star, star points=6, star point ratio=2, inner sep=0.15em, violet, fill, minimum size=3pt] at (11.8,4.7) {};
\node [star, star points=6, star point ratio=2, inner sep=0.15em, violet, fill, minimum size=3pt] at (13,5.05) {};
\node [star, star points=6, star point ratio=2, inner sep=0.15em, violet, fill, minimum size=3pt] at (14.2,5.4) {};
\node [star, star points=6, star point ratio=2, inner sep=0.15em, violet, fill, minimum size=3pt] at (15.4,5.7) {};

\node [star, star points=6, star point ratio=2, inner sep=0.15em, violet, fill, minimum size=3pt] at (11.8,3.9) {};
\node [star, star points=6, star point ratio=2, inner sep=0.15em, violet, fill, minimum size=3pt] at (13,4.2) {};
\node [star, star points=6, star point ratio=2, inner sep=0.15em, violet, fill, minimum size=3pt] at (14.2,4.5) {};
\node [star, star points=6, star point ratio=2, inner sep=0.15em, violet, fill, minimum size=3pt] at (15.4,4.8) {};

\end{tikzpicture}
\caption{\centering The points of $Y$ with nontrivial stabiliser under the action of $D_8$ (case 1)} \label{fig:M1}
\end{figure}

    \item \textbf{case 2:} The invariant subspaces of the action of $\beta$ on $\PP^5$ that are isolated points coincide with the points $p_5,p_6$ of \textbf{case 1}. Denote $\ell_1,\ell_2$ the two $\PP^1$ invariant subspaces of $\beta$: each one cuts $Y$ in two singular points and two regular points; the four regular points are all contained in the quadric surface $Q_\beta$ fixed by $\beta^2$. The four quadric surfaces $Q_k$ (fixed each by an involution in the copy of $D_8$ where $\beta$ is normal) intersect in two singular points lying on $\ell_1$; the preimage of the two singular points in $\ell_2$, together with the preimages of $p_5,p_6$, are the isolated points fixed by $\beta$ on $X$. the intersection $Q_1\cap Q_3$ contains two more singular points, that do not lie in any invariant subspace of $\beta$, so that their preimages in $X$ are fixed by $(\ZZ_2)^2$ (similarly $Q_2\cap Q_4$). 
    
    As above, considering the stabilisers in $L_3(4)$ we see there are actually four copies of $D_8$ acting on the same set of points: their elements of order 4 are conjugated to $\beta$, and other than $\beta^2$ there are 8 different involutions at play, so that the copies of $D_8$ can be paired in ``complementaries'' (with only the central involution in common). The singular points on $\ell_2$ belong to quadrics in the $D_8$ complementary to the first one; the roles of the singular points in $\ell_1,\ell_2$, and in the intersections $Q_1\cap Q_3,Q_2\cap Q_4$  are interchanged for the other pair of complementary copies of $D_8$. As in \textbf{case 1}, also in this case we find 8 points on $X$ fixed by $D_8$, and 8 points fixed by $\ZZ_4$ lying on the K3 surface $F_\beta$ double cover of $Q_\beta$.
    \end{itemize}

\begin{figure}
\centering
\begin{tikzpicture}[scale=0.6]
%hyperboloid (quadric Q_\beta)
  \node[draw, thick, red, dashed, ellipse, minimum width=2.5cm, minimum height=1cm, outer sep=0] (ell) at (17.7,5) {};
\draw[red, thick] (ell.-4) 
    to[out=-120,in=110](19.7,0) % (11,0) 
    to[out=-100,in=-70, looseness=.6] (15.8,0)
    to[out=70,in=-70] (ell.182);
\node[below, red] at (15.5,0) {$\scriptstyle Q_\beta$};

\draw[blue, thick] (4,4) to (21,4); %top line
\node[left, blue] at (4,4) {$\scriptstyle \ell_1$};

\draw[blue, thick] (4,1) to (21,1); %bottom line
\node[left, blue] at (4,1) {$\scriptstyle \ell_2$};
  % Coordinates
  \coordinate (A) at (6.00, 4.00); %d1
  \coordinate (R) at (6.00, 4.00);
  \coordinate (B) at (11.5, 4.00); %d2
  \coordinate (Q) at (11.5, 4.00);  
  \coordinate (BB) at (15.9, 4.00); %reg.pts
  \coordinate (BC) at (19.4, 4.00); 
  \coordinate (II) at (23, 3.50); %isolated pt
  
  \coordinate (RR) at (6.00, 1.00); %points on bottom line
  \coordinate (QQ) at (11.5, 1.00); 
  \coordinate (BD) at (16.1, 1.00); %reg.pts
  \coordinate (BE) at (19.4, 1.00); 
  \coordinate (JJ) at (23, 1.5); %isolated pt
  
  \coordinate (E) at (8.00, 5.3); %f1
  \coordinate (H) at (9.25, 5.35); %f2  
  \coordinate (X) at (8.3, 2.65); %g1
  \coordinate (U) at (9.5, 2.65); %g2
  
  %C,D: controls for purple top
  \coordinate (C) at (7.75, 5.75); 
  \coordinate (D) at (9.3, 6.00); 
  %M,N: controls for purple bottom
  \coordinate (M) at (7.75, 4.75);
  \coordinate (N) at (9.75, 4.75);

%F,G: controls for orange bottom 1
  \coordinate (F) at (7.00, 5.25);
  \coordinate (G) at (7.25, 6.50);
%I,J: controls for orange bottom 2
  \coordinate (I) at (8.4, 5.8);
  \coordinate (J) at (8.85, 5.8);
%I,J: controls for orange bottom 3  
  \coordinate (K) at (10.25, 6.5);
  \coordinate (L) at (10.5, 5.25);
%O,P: controls for orange top
  \coordinate (O) at (6.25, 7.5);
  \coordinate (P) at (11.50, 7.5);
  
%AC,AD: controls for green top  
  \coordinate (AC) at (9.8, 3.25);
  \coordinate (AD) at (7.8, 3.25);
%S,T: controls for green bottom
  \coordinate (S) at (9.8, 2.25);
  \coordinate (T) at (8.2, 2.00);  
  
%V,W: controls for blue top 3  
  \coordinate (V) at (10.5, 2.75);
  \coordinate (W) at (10.3, 1.50);
%Y,Z: controls for blue top 2  
  \coordinate (Y) at (9.1, 2.2);
  \coordinate (Z) at (8.7, 2.2);
%AA,AB: controls for blue top 1  
  \coordinate (AA) at (7.2, 1.5);
  \coordinate (AB) at (7., 2.75);
%AE,AF: controls for blue bottom
  \coordinate (AE) at (11.3, 0.5);
  \coordinate (AF) at (6, 0.5);

  % Drawing
  %Q3
  \draw[black, thick] (A) .. controls (C) and (D) .. (B); %top
  \draw[black, thick] (A) .. controls (M) and (N) .. (B); %bottom
  \node[above right, black] at (B) {$\scriptstyle Q_3$};

  %Q1
  \draw[red, thick] (A) .. controls (F) and (G) .. (E); %bottom1
  \draw[red, thick] (E) .. controls (I) and (J) .. (H); %bottom2
  \draw[red, thick] (H) .. controls (K) and (L) .. (B); %bottom3
  \draw[red, thick] (A) .. controls (O) and (P) .. (B); %top
  \node[above right, red, xshift=7pt] at (L) {$\scriptstyle Q_1$}; 
  
  %Q4
  \draw[red, thick] (Q) .. controls (S) and (T) .. (R); %bottom
  \draw[red, thick] (Q) .. controls (AC) and (AD) .. (R); %top
  \node[below right, red] at (Q) {$\scriptstyle Q_4$};

  %Q3
  \draw[black, thick] (Q) .. controls (V) and (W) .. (U); %t3
  \draw[black, thick] (U) .. controls (Y) and (Z) .. (X); %t2
  \draw[black, thick] (X) .. controls (AA) and (AB) .. (R); %t1
  \draw[black, thick] (Q) .. controls (AE) and (AF) .. (R); %bottom
  \node[above right, black, xshift=3pt] at (W) {$\scriptstyle Q_2$};

\node [star, star points=6, star point ratio=2, inner sep=0.15em, violet, fill, minimum size=3pt] at (E) {};
\node [star, star points=6, star point ratio=2, inner sep=0.15em, violet, fill, minimum size=3pt] at (H) {};
\node [star, star points=6, star point ratio=2, inner sep=0.15em, violet, fill, minimum size=3pt] at (R) {};
\node [star, star points=6, star point ratio=2, inner sep=0.15em, violet, fill, minimum size=3pt] at (Q) {};
\node [star, star points=6, star point ratio=2, inner sep=0.15em, violet, fill, minimum size=3pt] at (RR) {};
\node [star, star points=6, star point ratio=2, inner sep=0.15em, violet, fill, minimum size=3pt] at (QQ) {};
\node [star, star points=6, star point ratio=2, inner sep=0.15em, violet, fill, minimum size=3pt] at (X) {};
\node [star, star points=6, star point ratio=2, inner sep=0.15em, violet, fill, minimum size=3pt] at (U) {};

\fill[violet] (BB) circle (4pt);
\fill[violet] (BC) circle (4pt);
\fill[violet] (BD) circle (4pt);
\fill[violet] (BE) circle (4pt);
\fill[blue] (II) circle (4pt);
\fill[blue] (JJ) circle (4pt);
\node[right, blue] at (II) {$\scriptstyle p_5$}; 
\node[right, blue] at (JJ) {$\scriptstyle p_6$};  

\end{tikzpicture}
\caption{\centering The points of $Y$ with nontrivial stabiliser under the action of $D_8$ (case 2)} \label{fig:M2}
\end{figure}
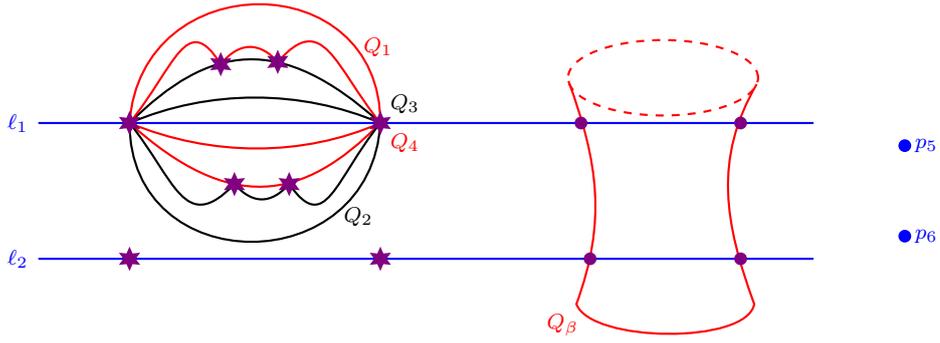

\subsubsection{The action of $D_{10}$}\label{sec:D10} The group $D_{10}$ contains 5 involutions $\iota_1,\dots,\iota_5$ (the reflections over symmetry axes of a pentagon), and the product of any two of them gives an element $g_5$ of order 5 or one of its powers.
Each $\iota_k$ fixes a $\PP^3$ and a $\PP^1$ in $\PP^5$ (by Remark \ref{fixed_locus_involution}), while the invariant subspaces for the action of $g_5$ on $\PP^5$ are $\PP^1_g$ and 4 isolated points $q_1,\dots,q_4$; it holds 
\[\bigcap_k \PP^3_k=\PP^1_g,\] 
and the $\PP^1_k$ do not intersect any other invariant subspace, nor do they intersect each other. Recall from~\cite{MongardiThesis} that a symplectic automorphism $g$ of order 5 on a K3$^{[2]}$-type manifold fixes 14 isolated points. The 14 points fixed by $g$ on $X$ are obtained from the intersection of the invariant subspaces of $g_5$ with $Y$ as follows: 8 as preimage in $X$ of $q_1,\dots,q_4$, the remaining 6 as preimage of $\PP^1_g\cap Y$, that consists of two regular points $r_1,r_2$ and two singular points $s_1,s_2$. 

Moreover, consider the five quadric surfaces $Q_k\subset \PP^3_k$: then $\PP^1_g\cap\big(\bigcap_k Q_k\big)=\{s_1,s_2\}$, so these two points are fixed by the whole group $D_{10}$, while the two regular points are contained in the quartic surfaces $K_k\subset\PP^3_k$, so by Remark \ref{fixed_locus_involution}
they give rise to 4 points on $X$ such that the two points on a fiber are exchanged (and so, they are fixed only by the automorphism of order 5).

Therefore, on $X$ we get 2 points lying on the intersection of five K3 surfaces $F_k$ and fixed by $D_{10}$, and 12 isolated points fixed by $\ZZ_5$.

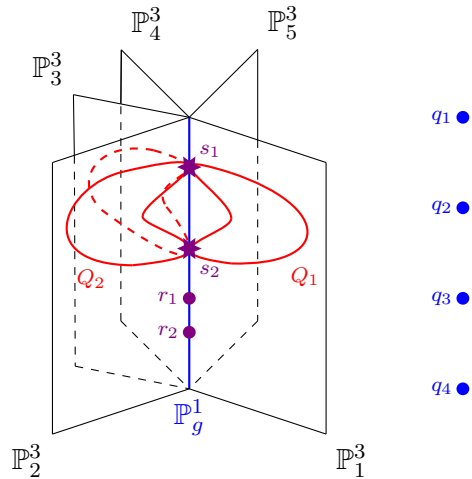
\begin{figure}
\centering
\begin{tikzpicture}[scale=0.6]
  % Coordinates
  \coordinate (A) at (2.00, 8.00);
  \coordinate (B) at (2.00, 2.00);
  \coordinate (C) at (5.00, 9.00);
  \coordinate (D) at (5.00, 3.00);
  \coordinate (E) at (8.00, 8.00);
  \coordinate (F) at (8.00, 2.00);
  \coordinate (G) at (2.46, 9.50);
  \coordinate (H) at (2.48, 8.15);
  \coordinate (I) at (2.50, 3.50);
  \coordinate (J) at (3.50, 4.50);
  \coordinate (K) at (3.50, 9.30);
  \coordinate (L) at (3.51, 10.49);
  \coordinate (M) at (6.50, 10.50);
  \coordinate (N) at (6.50, 8.50);
  \coordinate (O) at (6.50, 4.50);
  \coordinate (P) at (5.92, 7.77);
  \coordinate (Q) at (5.92, 5.77);
  \coordinate (R) at (8.15, 7.02);
  \coordinate (S) at (8.15, 5.71);
  \coordinate (T) at (5.92, 6.77);
  \coordinate (U) at (5.94, 6.43);
  \coordinate (V) at (3.89, 6.03);
  \coordinate (W) at (3.89, 8.42);
  \coordinate (X) at (5.92, 7.27);
  \coordinate (Y) at (4.22, 8.27);
  \coordinate (Z) at (3.15, 8.50);
  \coordinate (AA) at (2.40, 7.75);
  \coordinate (AB) at (3.03, 6.99);
  \coordinate (AC) at (4.50, 7.02);
  \coordinate (AF) at (6.00, 7.75);
  \coordinate (AG) at (4.00, 7.75);
  \coordinate (AD) at (5.50, 5.75);
  \coordinate (AE) at (5.00, 5.50);
  \coordinate (AH) at (3.56, 7.80);
  \coordinate (AI) at (3.75, 5.73);
  \coordinate (AJ) at (2.25, 7.50);
  \coordinate (AK) at (1.50, 5.75);
  \coordinate (AL) at (3.98, 6.82);
  \coordinate (AM) at (3.88, 6.55);
  \coordinate (AN) at (6.50, 6.00);
  \coordinate (AO) at (6.50, 8.50);
  \coordinate (AP) at (4.00, 7.25);
  \coordinate (AQ) at (5.00, 7.90);
  \coordinate (AR) at (5.00, 6.10);
  \coordinate (AS) at (5.00, 5.00);
  \coordinate (AT) at (5.00, 4.25);
  \coordinate (C_1) at (5.00, 9.00);
  \coordinate (D_1) at (5.00, 3.00);
  \coordinate (AU) at (11.00, 9.00);
  \coordinate (AV) at (11.00, 7.00);
  \coordinate (AW) at (11.00, 5.00);
  \coordinate (AX) at (11.00, 3.00);
  \coordinate (AY) at (7.00, 5.90);
  \coordinate (AZ) at (2.85, 5.85);

  % Drawing
  \draw (A) -- (B);
  \draw (A) -- (C);
  \draw[blue, thick] (C_1) -- (D_1);
  \draw (B) -- (D);
  \draw (C) -- (E);
  \draw (E) -- (F);
  \draw (D) -- (F);
  \draw (G) -- (C);
  \draw (G) -- (H);
  \draw[dashed] (H) -- (I);
  \draw[dashed] (I) -- (D);
  \draw[dashed] (J) -- (D);
  \draw[dashed] (J) -- (K);
  \draw (L) -- (K);
  \draw (K) -- (L);
  \draw (L) -- (C);
  \draw (C) -- (M);
  \draw (M) -- (N);
  \draw[dashed] (N) -- (O);
  \draw[dashed] (O) -- (D);
  \draw[red, thick] (P) .. controls (8.15, 7.02) and (8.15, 5.71) .. (Q);
  \draw[red, thick] (T) .. controls (5.94, 6.43) and (3.89, 6.03) .. (Q);
  \draw[red, thick] (P) .. controls (3.89, 8.42) and (5.92, 7.27) .. (T);
  \draw[red, dashed, thick] (Y) .. controls (3.15, 8.50) and (2.40, 7.75) .. (AB);
  \draw[red, dashed, thick] (Y) .. controls (6.00, 7.75) and (4.00, 7.75) .. (AC);
  \draw[red, dashed, thick] (AC) .. controls (5.50, 5.75) and (5.00, 5.50) .. (AB);
  \draw[red, thick] (AH) .. controls (2.25, 7.50) and (1.50, 5.75) .. (AI);
  \draw[red, thick] (AL) .. controls (3.88, 6.55) and (6.50, 6.00) .. (AI);
  \draw[red, thick] (AH) .. controls (6.50, 8.50) and (4.00, 7.25) .. (AL);
   \node [star, star points=6, star point ratio=2, rotate=30, inner sep=0.15em, violet, fill, minimum size=3pt] at (AQ) {};
  \node[above right, violet] at (AQ)  {$\scriptstyle s_1$};
  \node [star, star points=6, star point ratio=2, rotate=30, inner sep=0.15em, violet, fill, minimum size=3pt] at (AR) {};
  \node[below right, violet,yshift=-2.5pt] at (AR) {$\scriptstyle s_2$};
  \fill[violet] (AS) circle (4pt);
  \node[left, violet] at (AS) {$\scriptstyle r_1$};
  \fill[violet] (AT) circle (4pt);
  \node[left, violet] at (AT) {$\scriptstyle r_2$};
  \fill[blue] (AU) circle (4pt);
  \node[left, blue] at (AU) {$\scriptstyle q_1$};
  \fill[blue] (AV) circle (4pt);
  \node[left, blue] at (AV) {$\scriptstyle q_2$};
  \fill[blue] (AW) circle (4pt);
  \node[left, blue] at (AW) {$\scriptstyle q_3$};
  \fill[blue] (AX) circle (4pt);
  \node[left, blue] at (AX) {$\scriptstyle q_4$};
  \node[below left, black] at (B) {$\mathbb P^3_2$};
  \node[below right, black] at (F) {$\mathbb P^3_1$};
  \node[above left, black] at (G) {$\mathbb P^3_3$};
  \node[above right, black] at (L) {$\mathbb P^3_4$};
  \node[above right, black] at (M) {$\mathbb P^3_5$};
  \node[below, blue] at (D_1) {$\mathbb P^1_g$};
  \node[below right, red] at (AY) {$\scriptstyle Q_1$};
  \node[below, red] at (AZ) {$\scriptstyle Q_2$};
\end{tikzpicture}
\caption{\centering The points of $Y$ with nontrivial stabiliser under the action of $D_{10}$} \label{fig:M3}
\end{figure}

\subsubsection{Points with bigger stabiliser}\label{sec:points_bigger_stabiliser} We can check by computer (using~\cite{M2, GAP}) whether there are points fixed by non-cyclic subgroups $H\subset L_3(4)$ bigger than $S_3,D_8,D_{10}$. Indeed, notice that all possible such $H$ contain some subgroup isomorphic to one of the groups in Lemma \ref{groups_gen_by_involutions} (see Appendix \ref{Ap:fix}), and we can apply Lemma \ref{lem:orbit_of_points}.
The only case of a bigger stabiliser we found is one orbit of regular points $p\in Y$ fixed by a group $\tilde H$ of order 32: this is the group with Id (32, 31) in GAP's small group library. This orbit contains the points $p_5,p_6\in Y$, found by studying the action of $D_8$ (we will prove they are mapped one to the other by the action of $L_3(4)$ in Proposition \ref{prop:non_isolated_sings}). As already discussed in Remark \ref{rem:X_and_Y}, the two points of $X$ belonging to the fiber $\pi^{-1}(p)$ may have a smaller stabiliser than that of $p$ (a subgroup of index 2 to be precise), as the lift of the action of the stabiliser group from $Y$ to $X$ may exchange them. Indeed, we shall prove this happens here.
\begin{thm}
    Let $p\in Y$ be a point with stabiliser $\tilde H$, let $q_1,q_2$ be the points in $\pi^{-1}(p)\subset X$. Then the stabiliser of $q_i$ is $Q_8\times\ZZ_2$.
\end{thm}
\begin{proof}
    The point $p$ belongs to exactly two quadric surfaces $Q_1,Q_2\subset Y$ fixed by involutions. 
    Assume that the action of $\tilde H$ lifts to $X$: then, we take the quotient of $\tilde H$ by one of the two involutions $\iota$ that fix a quadric surface around $p$, to get the residual action of $\tilde H$ on each quadric surface (and, therefore, on its K3 double cover). The result is $\tilde H/\iota=Q_8\ast\ZZ_4$; while this group can act symplectically on a K3, it is never the stabiliser of a point (only its subgroups fix points):
    this follows from the classification of symplectic surface singularities (\cite[Satz 2.9]{Brieskorn}, see also~\cite[Proposition 2.2]{Marisia}), compared with the description of the exceptional curves for this quotient in~\cite[Table 2, row 24]{Xiao}.
    
    Therefore, we can conclude that a subgroup of $\tilde H$ of index 2 is fixing the preimages $q_1,q_2$ of $p$ on $X$: the possible subgroups $K\subset \tilde H$ are $\ZZ_4^2,\ (\ZZ_4\times \ZZ_2):\ZZ_2,\ \ZZ_2\times D_8,\ \ZZ_2\times Q_8$. Since each $q_i$ is contained in two K3 surfaces, and the local action of $K$ on them around it is the same, we have the following constraints: $K$ has to contain two normal involutions that give isomorphic quotient groups, and said quotient group should stabilise a point on a K3 surface. The only group that satisfies this requirements is $K=\ZZ_2\times Q_8$ (the group fixing $q_i$ on each K3 surface is $K/\iota=Q_8$). 
\end{proof}

\section{Singularities of $\texorpdfstring{X/L_3(4)}{X/L3(4)}$}\label{sec:singularities}
The aim of this section is to classify the singularities of the quotient $X/L_3(4)$. This involves studying the action of the group $L_3(4)$ on the points of $X$ with nontrivial stabiliser found in the previous section, using again the model of the EPW sextic $Y\subset\mathbb P^5$ and the description of the group action provided in Appendix \ref{Ap:codes}, together with a careful analysis of how this action lifts to the \HK double cover $X\rightarrow Y$, as explained in Remark \ref{rem:preimages_are_fixed_by_iota}.

\begin{thm}\label{main2}
    The quotient $X/L_3(4)$ is singular along a surface $S$ with generally transversal $A_1$ singularities along it; its special singular points are one point of type $\CC^4/Q_8\times \ZZ_2$, two points of type $\CC^4/D_{10}$, three of type $\CC^4/S_{3}$, three of type $\CC^4/D_{8}$ and five of type $\frac{1}{4}(1,3,2,2)$.
    Moreover, $X/L_3(4)$ has six isolated singularities of order 3, twelve of order 5, nine of order 7.
    It is smooth everywhere else.
\end{thm}   
\begin{proof}
    See Propositions \ref{prop:non_isolated_sings} and \ref{prop:isolated_sings}.
\end{proof}

\subsection{Non-isolated singularities}
The quotient $X/L_3(4)$ is singular along a surface $S$, the image in the quotient of the K3 surfaces $F_k$, each fixed by one of the 315 conjugate involutions $\iota_k\in L_3(4)$.
The singularity of the quotient at a point $p \in S$ is of the shape
$\CC^4/G$, where $G$ contains a normal involution $\iota$ such that, in the local representation of the action of $G$ around $p$, $\iota$ has a 1-eigenspace of dimension 2. For the general point $p\in S$ it holds $G=\langle\iota\rangle$, so that $p$ is an ordinary doube point.

\begin{prop}\label{prop:non_isolated_sings}
      The following singularities of $X/L_3(4)$ are at points on $S$:
    \begin{enumerate}
        \item 5 points of the form $\CC^4/\ZZ_4$;
        \item 3 points of the form $\CC^4/D_8$;
        \item 1 point of the form $\CC^4/Q_8\times\ZZ_2$; 
        \item 3 points of the form $\CC^4/S_3$;
        \item 2 points of the form $\CC^4/D_{10}$.

    \end{enumerate}
\end{prop}
\begin{proof}
     In Section \ref{sec: subgroups_stabilising_points} we computed the action on the EPW sextic $Y$ (and on its double cover $X$) of groups $G\subset L_3(4)$ generated by two involutions. For each point $p\in Y$ with a non-trivial stabiliser, we also determined if its pre-image $\pi^{-1}(p)\subset X$ is contained in some K3 surface $F_k$ fixed by an involution: this holds if and only if $p\in Q_k\subset Fix_Y(\iota_k)$, where $Q_k=\pi(F_k)$ is a quadric surface.
     The last step consists of checking (using~\cite{M2}) if any of the selected points lie in the same orbit for the action of $L_3(4)$: this is especially important for points fixed by a group containing an element of order 4, as these belong to one of 3 different conjugacy classes.
     
     Recall that, denoting $\alpha,\beta,\gamma$ a choice of representatives for these classes, the action of $\alpha$ on $Y$ is studied in Section \ref{sec:D8_on_Y}, case 1, while that of $\beta,\gamma$ in case 2.
     Consider case 1: then the regular points $p_1,\dots,p_4$ fixed by $D_8\times \ZZ_2$ are all identified by the action of $L_3(4)$. The same holds for the two regular points $p_5,p_6$ with stabilizer of order 32, and all the singular points fixed by $\ZZ_4$ on the quadric surface fixed by $\alpha^2$.
     Moreover, the points on $X$ of the fibres over $p_1,\dots, p_6$ are pairwise identified by Remark \ref{rem:X_and_Y}: indeed, these points are fixed on $X$ by a subgroup of index 2 of the group fixing their image in $Y$.
     Therefore on $X$ we get one orbit of points fixed by each of the groups $\ZZ_4,D_8,Q_8\times \mathbb Z_2$.
     
     Similarly, on $Y$ case 2 gives (other than the same two points with stabiliser of order 32 discussed above) 4 singular points fixed by $D_8$, and 4 regular points fixed by $\ZZ_4$, but again the points fixed by the same group get identified by the action of $L_3(4)$: therefore, on $X$ we get one orbit of points fixed by $D_8$, and two orbits of points fixed by $\ZZ_4$. Since this happens for two different conjugacy classes of elements (represented by $\beta$ and $\gamma$), we proved points $(1)-(3)$ of the statement.
     
     All elements $g\in L_3(4)$ of order 3 (and all copies of $S_3$) are conjugate: the invariant subspaces for the action of $g$ on $\PP^5$ are three copies of $\PP^1$, and there exist elements of $L_3(4)$ mapping each one to each other; on each of these $\PP^1$, three points are fixed by some copy of $S_3$. Therefore, there are only three singularities of the form $\CC^4/S_3$ in the quotient. On the other hand, the invariant subspaces for an action of an element $h$ of order 5 are a copy of $\PP^1$ and 4 isolated points, and there is no element of $L_3(4)$ acting non-trivially on this set: in particular, the two singular points fixed by the $D_{10}$ in which $h$ is normal, that lie on the $\PP^1$-invariant subspace, are in different orbits for the action of $L_3(4)$.
\end{proof}

\subsection{The local representations}\label{representations}

 We now describe the local representation of those groups $G$ appearing in Proposition \ref{prop:non_isolated_sings}. As already discussed in Remark \ref{rem:natural_action}, the action of any of those groups can be deformed to an action induced from a K3 surface $S$: suppose therefore that on $S^{[2]}$ the fixed point is $p=P_1+P_2$, where $P_1\neq P_2$ are fixed on $S$ by some subgroups of $G$.
If the local action around $P_i$ is of the shape $(x,y)\to (e^{a_i}x,e^{b_i}y)$, then the action around $p$ is $(x,y,z,t)\to (e^{a_1}x,e^{b_1}y,e^{a_2}z,e^{b_2}t)$; if moreover the quotient has symplectic singularities, then $a_i+b_i=0$. Therefore, we deduce the following local representations:
\begin{itemize}
\item  The local action of $\ZZ_4$ is generated by:
    \begin{equation}\label{eq:local_Z4}
        \begin{bmatrix}
        i &0 &0 &0\\ 0 &-1 &0 &0\\ 0 &0 &-i &0\\ 0 &0 &0 &-1
    \end{bmatrix}.
    \end{equation}
\item  The local action of $S_3,D_8,D_{10}$ is generated by:
        \begin{equation}\label{eq:local_Dn}
        \begin{bmatrix}
        \zeta &0 &0 &0\\ 0 &\zeta^{-1} &0 &0\\ 0 &0 &\zeta^{-1} &0\\ 0 &0 &0 &\zeta
    \end{bmatrix},\quad \begin{bmatrix}
        0 &0 &1 &0\\ 0 &0 &0 &1\\ 1 &0 &0 &0\\ 0 &1 &0 &0
    \end{bmatrix},
    \end{equation}
    where $\zeta$ is respectively a third, fourth or fifth primitive root of unity.
\item  The local action of $\ZZ_2\times Q_8$ is generated by:
\begin{equation}\label{eq:local_Z2Q8}
     \begin{bmatrix}
     0 &0 &i &0\\ 0 &0 &0 &i\\ i &0 &0 &0 \\0 &i &0 &0
    \end{bmatrix},\quad \begin{bmatrix}
    i  &0 &0 &0\\ 0 &i &0 &0\\ 0 &0 &-i &0\\ 0 &0 &0 &-i
    \end{bmatrix},\quad \begin{bmatrix}
    1 &0 &0 &0\\ 0 &-1 &0 &0\\ 0 &0 &1 &0 \\ 0 &0 &0 &-1\end{bmatrix}
    \end{equation}
\end{itemize}

\subsection{Isolated singularities}

If $p$ is an isolated singularity of $X/L_3(4)$, that is $p\notin S$, the stabiliser of $p$ has to be one of the following groups: $\ZZ_2,\ZZ_3,\ZZ_5,\ZZ_7$. Indeed, the other subgroups $H$ of $L_3(4)$ with non-empty fixed locus contain involutions, and we can check using~\cite{M2} that points fixed by $H$ always belong to some surface fixed by some of its involutions.

\begin{prop}
    There are no isolated points on X whose stabiliser is $\ZZ_2$.
\end{prop}
\begin{proof}
The fixed locus of an involution is described in Remark \ref{fixed_locus_involution}. Every involution $\iota$ is the square of some elements of order 4: in particular, we can select representatives $\alpha, \beta,\gamma$ of the three conjugacy classes of elements of order 4, such that $\alpha^2=\beta^2=\gamma^2=\iota$.

The isolated points fixed by $\iota$ on $Y$ are 6 regular points on a line $\ell$, and 16 singular points on a quartic surface $K\in\PP^3$. Let $q\in X$ be such that $\pi(q)\in\ell$ is fixed by $\iota$: then, the stabiliser of $q$ is either $D_8$ or $Q_8\times\ZZ_2$, and in particular $q$ is also contained in some K3 fixed by an involution (see Section \ref{sec:D8_on_Y}, case 1: $\ell$ is also an invariant subspace for the action of $\alpha$, and all six points are contained in some quadric surfaces). The 16 points $q\in X$ such that $\pi(q)\in K$ also have stabiliser $D_8$: indeed, we can pair them in such a way that each pair lies in the intersection of 4 quadric surfaces $Q_k$, fixed by involutions that generate a copy of $D_8$ (see Section \ref{sec:D8_on_Y}, case 2: notice that the central element of $D_8$ can be conjugated to either $\beta$ or $\gamma$).
\end{proof}

\begin{prop}\label{prop:isolated_sings}
    The quotient $X/L_3(4)$ has the following isolated (terminal) singularities:
    \begin{itemize}
        \item 6 of type $\frac{1}{3}(1,-1,1,-1)$;
         \item 12 points of type $\frac{1}{5}(1,-1,2,-2)$; 
        \item 9 of type $\frac{1}{7}(1,-1,2,-2)$;
    \end{itemize}
\end{prop}
\begin{proof}
A singularity of $X/L_3(4)$ is terminal if and only if it is not contained in the singular surface $S$, image of the K3 surfaces fixed on $X$ by the 315 involutions.

An element $g_3$ of order 3 always belongs to a copy of $S_3$: of the 27 points fixed by $g_3$, 9 are fixed by $S_3$ (see Section \ref{sec:S3}, and Proposition \ref{prop:non_isolated_sings} for the singularities they produce); using~\cite{M2}, we can check that other elements of order 3 in $L_3(4)$ act on the remaining 18 points, so the quotient manifold has 6 isolated singularities of order 3. Similarly, an element $g_5$ of order 5 always belongs to a copy of $D_{10}$: of the 14 points fixed by $g_5$, 2 are fixed by $D_{10}$, but this time the 12 remaining points are not identified in the quotient. The elements of order 7 are not contained in a bigger group, and they give $9$ isolated fixed points on $X$ (see also Lemma \ref{lem:orbit_of_points}).

In order to compute the weights of the quotient singularities we shall use the fact,
observed in~\cite{Anno}, that if $G=\ZZ_n$ and $\CC^4/G$ is a terminal Gorenstein singularity, then its type is $\frac{1}{n}(1,-1,a,-a)$, for some weight $0<a<n$ such that $(a,n)=1$. For $n=3,5$ and $G$ acting naturally on a K3$^{[2]}$-type manifold $X$, the weights of the singularities of $X/G$  have been determined in~\cite[Chapter 6]{MongardiThesis} using the holomorphic Lefschetz-Riemann-Roch formula. We can use the same strategy also in the order 7 case. Notice first that, up to isomorphisms, the only two possibilities are $\frac{1}{7}(1,-1,1,-1)$ and $\frac{1}{7}(1,-1,2,-2)$. 
Denote the order 7 automorphism by $\phi$: the Lefschetz-Riemann-Roch formula for the sheaf $\mathcal{O}_X$ gives 
     \begin{align*}
         3 = \sum_{p \in \text{Fix}(\phi)} \frac{1}{ct(\bigoplus_t\bigwedge^tN^*_p, \phi_p)},
     \end{align*}
where $ct$ is the Chern trace, and $N_p$ is the normal sheaf at a fixed point $p$ of order 7. There are six possible values for the summands on the right hand side: we can subdivide them in two sets of three $\{a_i\}$ and $\{b_i\}$, depending on whether the action of $\phi$ on the tangent space near $p$ can be written as a diagonal matrix $diag(\omega, \omega^{-1}, \omega, \omega^{-1})$, or $diag(\omega, \omega^{-1}, \omega^2, \omega^{-2})$ respectively, where $\omega$ is a primitive root of unity of order 7 (that changes depending on $i$). More precisely, fixing $\zeta$ a primitive root of unity of order 7, we can write each of the $a_i, b_i$ in the form $-\sum_{j=1}^6\frac{c_j}{7}\zeta^j$, with the coefficients $c_j$ as in the following tables:
\begin{center}
\begin{tabular}{c|cccccc}
     \  & $c_1$ &$c_2$ &$c_3$ &$c_4$ &$c_5$ &$c_6$\\
     \hline
     $a_1$&  2 &5 &7 &7 &5 &2\\
     $a_2$&  7 &2 &5 &5 &2 &7\\
     $a_3$&  5 &7 &2 &2 &7 &5\\
\end{tabular} \quad
\begin{tabular}{c|cccccc}
     \  & $c_1$ &$c_2$ &$c_3$ &$c_4$ &$c_5$ &$c_6$\\
     \hline
     $b_1$&  2 &3 &2 &2 &3 &2\\
     $b_2$&  3 &2 &2 &2 &2 &3\\
     $b_3$&  2 &2 &3 &3 &2 &2\\
\end{tabular}
\end{center}
 We can then directly check that there exists only one solution to the equation
    \begin{align*}
         3 = \sum_{i=1}^3 n_ia_i + \sum_{i=1}^{3}m_ib_i,
     \end{align*}
     where $n_i$ and $m_i$ are nonnegative integer summing up to nine, this solution being $n_i=0$, $m_i=3$ for all $i$.
\end{proof}

\section{Terminalisations}\label{sec:terminalisations}

In this section, we are going to describe the terminalisation $\tilde X$ of $X/L_3(4)$: this quotient indeed is not terminal, as it has a codimension~2 locus of points with non-trivial stabiliser, i.e. the surface $S$ obtained as quotient of the K3 surfaces $F_k\subset X$ fixed by conjugated involutions $\iota_k$. The general point of $S$ is an ordinary double point; the special points of $S$ are enumerated in Proposition \ref{prop:non_isolated_sings}, and require to be dealt with separately (the $\QQ$-factorial terminalisation will be performed in local analytic neighborhoods). 

We remark that by~\cite[Proposition 1.1]{Menet33} $\tilde{X}$ is an irreducible symplectic variety; the residual singularities after the terminalisation are rigid, and occur on a general element deformation equivalent to $\tilde X$ (cf.~Remark \ref{rem:dimension_of_moduli_tildeX}). 

\begin{defi}
    A $\QQ$-factorial terminalisation of a normal variety $X$ with symplectic singularities is a proper birational morphism $f\colon\tilde X\to X$ such that $\tilde X$ has only terminal singularities, and $f$ is crepant, i.e. $K_{\tilde X}=f^*K_X$.
\end{defi}

\begin{rem}\textit{Gluing local terminalisations.}
In general, the local terminalisation of a symplectic singularity is not unique, as many birational models can exist (see for instance Proposition \ref{prop:S3_D8}); this might present a problem when gluing local terminalisations to a global one. However, we remark that in our case, if $p\in S$ is a special singularity, still in any local analytic neighbourhood of $p$ we are going to find ordinary double points: as the latter admit a unique terminalisation (actually, a smooth resolution), this ensures correct gluings around the special fibres. The exceptional divisor of our terminalisation is going to be a generically $\PP^1$-fibration over $S$, but it 
will indeed have some special fibres which are surfaces (again, see Proposition \ref{prop:S3_D8}).

We also remark that, by~\cite[Theorem 6.16]{BL}, birational ISVs are deformation equivalent. 
\end{rem}

\subsection{The group $\mathbb Z_4$}
There are five orbits of points of $X$ fixed by $\mathbb Z_4=\langle\alpha\rangle$, with local action as in \eqref{eq:local_Z4}: for these points, a single blow-up is still enough to get a terminalisation.

In the local coordinates $(x_1,x_2,x_3,x_4)$ the surface $F_\alpha$, fixed by $\alpha^2$, is described by $x_1=x_3=0$. 
The blow-up $\tilde F_\alpha$ of $\CC^4$ along $F_\alpha$ is described by $\{(y_0:y_1)(x_1,x_2,x_3,x_4)\mid y_0x_3=y_1x_1\} \subseteq \mathbb P^1\times\mathbb C^4$: therefore, the lift of $\alpha$ to the blow-up acts on the exceptional fibre over the point $(0,0,0,0)$ fixed by $\alpha$ on $F_\alpha$, mapping $(y_0:y_1)\mapsto (y_0:-y_1)$, and thus fixing the points $(1:0),(0:1)$. Now, the quotient of $\tilde F_\alpha$ by $\alpha^2$ is smooth, with a residual action of $\alpha$ as an involution fixing the two points over the origin: being isolated points of order 2, these become terminal singularities after the quotient by $\alpha$.

Therefore, the terminalisation of the 5 points of order 4 on the singular surface $S\subset X/L_3(4)$ can be obtained by a blow-up, and the residual singularities are 10 points of type $\frac{1}{2}(1,1,1,1)$.

Note that the quotient $\mathbb{C}^4/\mathbb{Z}_4$ is a toric variety, hence its terminalisation is also toric and can be described in terms of a fan, using a criterion for terminality of toric singularities given in~\cite{Ghirlanda}.

\subsection{Computing terminalisations via Cox Rings}

Let $V$ be a 4-dimensional complex vector space. The aim of this section is to explain how to construct a terminalisation of $V/G$ for $G\subset SL(V)$ using Cox rings. This method was introduced in~\cite{DW}, where the authors find symplectic resolutions $W$ of a quotient singularity $V/G$ for $G=Q_8\times_2 D_8$, the central product of the quaternion and the 
dihedral group of order 8 over a common center of order 2 (GAP Id: (32,50)). The main idea is to construct the Cox ring $\mathcal R(W)$ of a resolution $W$ of $V/G$ indirectly, without understanding $W$ first. Any projective symplectic resolution can then be found as a GIT quotient of $\mathrm{Spec}(\mathcal R(W))$ by the action of the Picard torus $\Hom(\mathrm{Cl}(W), \CC^*)$.

In~\cite{Yam} it is shown that the same method applied to other groups $G\subset SL(V)$ produces minimal models of the singularity $V/G$, i.e. its possible terminalisations. 

\subsubsection{Outline of the method}\label{sec:outline_method}
We briefly describe here the construction of terminalisations of $V/G$ via Cox rings; more details can be found e.g. in~\cite{DW, Yam, DG}.
Assume that $W\rightarrow V/G$ is a $\QQ$-factorial terminalisation. As one may expect, the Cox ring $\mathcal{R}(W)$ can be described in terms of the Cox ring $\mathcal{R}(V/G)$ of the singularity and the exceptional divisors in the terminalisation. Indeed, there is an embedding
\begin{equation}\label{eq:emb1}
\mathcal R(W)\hookrightarrow  \mathcal R(V/G)\otimes_{\CC} \CC[\mathrm{Cl}(W)^{free}],
\end{equation}
where $\mathcal R(V/G)\simeq \CC[V]^{[G,G]}$ as $\CC[V]^G$-algebras graded by $\mathrm{Cl}(V/G)\simeq Ab(G)^\vee$.
Let $m$ be the number of components of the exceptional divisor of the terminalisation: then $\mathrm{Cl}(W)^{free} \simeq \ZZ^m$ and we can rewrite \eqref{eq:emb1} as 
\begin{equation}\label{eq:emb2}
\mathcal R(W)\hookrightarrow \CC[V]^{[G,G]}[t_1^{\pm1},\dots ,t_m^{\pm1}].
\end{equation}
There is a method, first proposed in~\cite{DW}, and further developed in~\cite{DG, Yam}, for searching for a set of elements of the ambient ring that generate the image of $\mathcal R(W)$. In particular, \cite{Yam} provides an algorithm to construct such a generating set (which, however, may not terminate in a reasonable time).

Assuming that $\mathcal{R}(W)$ is described as a subring of $\CC[V]^{[G,G]}[t_1^{\pm1},\dots ,t_m^{\pm1}]$ by a finite set of generators $f_1,\ldots, f_d$, one may try to recover $W$, and its other birational models, as a GIT quotient of $\Spec(\mathcal{R}(W))$ by the action of the Picard quasitorus $\Hom(\mathrm{Cl}(W), \CC^*)$ induced by the grading (see~\cite[Sec.~3.3]{ADHL}). This step can be performed if the obtained generating set is small enough to proceed with the computations. 

The Picard quasitorus can be written as $T\times K$, where $T\simeq (\CC^*)^m$ is a torus and $K$ is a finite group.
First, we take the GIT quotient by $T$ and then we analyze the induced action of $K$. The computations are performed in several steps, briefly described below; for the details we refer to~\cite[Sec.~4.A]{DW}.
\begin{enumerate}
    \item \emph{Pick a suitable element $\chi$ in the character lattice of~$T$.} This element should determine a linearization of the trivial bundle on $Z := \Spec(\mathcal{R}(W)) \subseteq \CC^d$ which corresponds to a minimal model of $V/G$, that is, it should lie in the interior of a Mori chamber in the cone of movable divisors. The chamber decomposition can be determined following~\cite[Proposition~2.9]{BerchtoldHausen}.
    \item \emph{Find the set $Z_{\chi}^{ss}$ of points in $Z$ semistable with respect to $\chi$.} Note that being $\chi$-semistable is a property of whole $(\CC^*)^d$ orbits on~$\CC^d$ and it restricts well to a closed subvariety.  
    We determine the set of semistable orbits, and check that all o them are actually stable, using Singular~\cite{Singular} code developed for~\cite{DW}.
    \item \emph{Find the singularities of the geometric quotient of $Z_{\chi}^{ss}$ by $T$.} If $Z_{\chi}^{ss}$ is smooth and the action of $T$ (or a quotient of $T$ by a finite group action) on the set of semistable points is free then by Luna's slice theorem \cite{Luna} the quotient is smooth. Thus we investigate the singularities of $Z_{\chi}^{ss}$, which in considered cases turns out too complex to be done directly by computing the singular set of $Z$. 
    However, it is possible to avoid direct computations and find monomials vanishing on the singular set of $Z$, which verifies the smoothness of some large semistable orbits and reduces the problem to investigating singularities of smaller ones. 
    \item \emph{Find the singularities of the induced action of $K$ on $Z_{\chi}^{ss}/T$.} We investigate the action of $K$ on $Z_{\chi}^{ss}$, determine the set of points with non-trivial isotropy groups and look at its behavior under the action of~$T$.
\end{enumerate}

\subsubsection{Junior elements}\label{junior}

\begin{defi}
Let $g \in SL_n(\mathbb{C})$ be an element of finite order~$r$. Let $\zeta$ be a primitive root of unity of order $r$, then $g$ has the diagonal form $diag(\zeta^{a_1},\ldots,\zeta^{a_n})$, where $a_1, \ldots, a_n \in \{0,...,r-1\}$. Then we define the \emph{age} of $g$ as 
\begin{align*}
    \text{Age}(g) = a_1 +\ldots+a_n \in \ZZ.
\end{align*}
If $\text{Age}(g) = 1$, the $g$ is called a \emph{junior}. 
\end{defi}

Junior conjugacy classes in $G$ correspond via McKay to irreducible components of the exceptional divisor of the terminalisation of $V/G$ (see~\cite[1.5]{ItoReid}, \cite{KaledinMcKay}).
By~\cite[Theorem 5.1]{SchmittClassGroup}, for a terminalisation $W\rightarrow V/G$ it holds 
\begin{equation}\label{eq:class_group_terminalisation}
    \mathrm{Cl}(W)\simeq \ZZ^m\oplus Ab(G/H)^\vee,
\end{equation}
where $m$ is the number of junior conjugacy classes, and $H$ is the subgroup of $G$ they generate. By~\cite[Theorem 1.1]{Yam} (see also~\cite{Verbitsky}, for the case of symplectic groups), $H=G$ is a necessary condition for $V/G$ to admit a resolution.

We remark that when $X$ is a \HK fourfold, junior elements are easily described:
\begin{prop}
    Consider $X$ an irreducible symplectic fourfold, $f$ a finite symplectic automorphism, $x$ a point fixed by $f$. Then on the germ $(X, x)$, we have either
    \begin{enumerate}
        \item $\Age(f) = 1$ ($f$ is a junior) and $f$ fixes a surface in $(X,x)$, or
        \item $\Age(f) = 2$ and $x$ is an isolated fixed point.
    \end{enumerate}
\end{prop}

\begin{proof}
    If $x$ is an isolated fixed point, consider the variety $Y = X/\langle f \rangle$, and $y$ the image of $x$ in $Y$. Then $Y$ is a symplectic variety and $y$ is an isolated symplectic singularity on a fourfold, so it must be terminal Gorenstein (because it is singular in codimension 4). Then by~\cite{Anno} $y$ is of type $\frac{1}{r}(1,-1,a,a)$ where $r$ is the order of $f$ and $a$ is coprime with $r$. Hence $\Age(f) = 2$. Conversely, assume $x$ is not an isolated fixed point. Then because $f$ is symplectic, $f$ must fix a surface near $x$. But that means that local action has eigenvalues $(\zeta, \zeta^{-1}, 1, 1)$, where $\zeta$ is a nontrivial root of unity, and so $\Age(f) = 1$.
\end{proof}
\begin{rem}
   The representations of $S_3,D_8, D_{10}$ in~\ref{representations} are generated by junior elements.
\end{rem}

\subsubsection{The groups $S_3$, $D_8$}
In~\cite{DG} the method described in Section \ref{sec:outline_method} is applied to the groups $S_3,D_8$, acting on $\CC^4$ as in \eqref{eq:local_Dn}, with the following outcome.
\begin{prop}\label{prop:S3_D8} By~\cite[Theorem 3.8, Theorem 5.21]{DG},
    \begin{enumerate}
        \item  $\CC^4/S_3$ admits a unique symplectic resolution.
        \item $\CC^4/D_8$ admits two symplectic resolutions that differ by a Mukai flop.
    \end{enumerate}
\end{prop}

\subsubsection{The group $\texorpdfstring{\mathbb{Z}_2\times Q_8}{Z2Q8}$}

We now apply the method described in Section \ref{sec:outline_method} to the singularity $\CC^4/G$, where $G\simeq\mathbb{Z}_2\times Q_8$ acts as described in \eqref{eq:local_Z2Q8}. We remark that the terminalisation of this singularity was not previously described in the literature.

The following result summarises the necessary data on $G$, and is proved by a direct computation in~\cite{GAP}.

\begin{lem}\label{lem:Z2xQ8_data}
The group $G$ has 10 conjugacy classes, where 2 are junior and generate a normal subgroup $H\subset G$, $H\simeq\ZZ_2^2$. The derived subgroup $[G,G] \subset G$ is $\ZZ_2$ generated by $-Id$, and $Ab(G)\simeq\ZZ_2^3$. The subgroup $H$ contains $-Id$, hence $G/H$  is abelian and isomorphic to $\ZZ_2^2$.
\end{lem}

Note that, since $G$ is not generated by junior elements, the quotient $V/G$ does not admit a crepant resolution.
A terminalisation $W$ of $V/G$ may be obtained as a GIT quotient of $Z = \Spec(\mathcal{R}(W))$ by the quasitorus $\Hom(\mathrm{Cl}(W), \CC^*)$, that by Lemma \ref{lem:Z2xQ8_data} and \eqref{eq:class_group_terminalisation} we can write as $T\times K:= (\mathbb{C}^*)^2\times \mathbb{Z}_2^2$.\\

To find the generating set of $\mathcal{R}(W)$ we start with a generating set of $\CC[V]^{[G,G]}$ consisting of eigenvectors of the induced action of $Ab(G)$. The ring of invariants is generated by all 10 monomials of degree~2 in~4 variables $x,y,z,w$ and we take the following combinations to make them eigenvectors of the action of $Ab(G)$:
\[xz,  yw, -x^2 + z^2, -y^2 + w^2, x^2 + z^2, y^2 + w^2, -yz+xw, -xy+zw, yz+xw, xy+zw\]
To have the first candidate for the generating set of the Cox ring (which, in this case, turns out to be a generating set) one has to modify the elements above by monomials in variables $t_1, t_2$ corresponding to junior classes, and add two generators related to the exceptional divisors, as given below.

\begin{prop}\label{prop:generators_R(W)}
The following subset of $\CC[V]^{[G,G]}[t_1^{\pm 1}, t_2^{\pm 1}]$ generates the Cox ring $\mathcal{R}(W)$ of a terminalisation of $V/G$:
    \begin{align*}
    &xzt_1^2,  ywt_2^2, (-x^2 + z^2)t_1^2, (-y^2 + w^2)t_2^2, (x^2 + z^2)t_1^2, (y^2 + w^2)t_2^2, \\
    &(-yz+xw)t_1t_2, (-xy+zw)t_1t_2, (yz+xw)t_1t_2, (xy+zw)t_1t_2, t_1^{-2}, t_2^{-2}
\end{align*}
and the weight matrix corresponding to the action of $T$ on these generators is
\[D=\begin{pmatrix}
    2 & 0 & 2 & 0 & 2 & 0 & 1 & 1 & 1 & 1 & -2 & 0\\
    0 & 2 & 0 & 2 & 0 & 2 & 1 & 1 & 1 & 1 & 0 & -2
\end{pmatrix}\]

\end{prop}
\begin{proof}
The weights in $D$ correspond to values of monomial valuations on the chosen $[G,G]$-invariants. They indicate how to modify them by monomials in $t_1,t_2$ to produce generators of $\mathcal{R}(W)$ and determine the action of $T$ on $\Spec(\mathcal{R}(W))$. 
To prove that this is a complete generating set, one has to check the valuation lifting condition from~\cite{DG}. We use the algorithm proposed by Yamagishi in~\cite{Yam}, implemented in~\cite{M2} and~\cite{Singular}. 
\end{proof}

To describe directly the geometric properties of GIT quotients of $Z$, in particular singular points, one has to look at the ideal $I_{\mathcal R}$ of relations between the generators $v_1,\dots v_{12}$ of $\mathcal{R}(W)$ (given in Proposition \ref{prop:generators_R(W)}), which describes an embedding $Z \subset \CC^{12}$. It can be computed directly using~\cite{Singular}:

\begin{align*}%\lvbel{eq:I_R}
%\begin{split}
I_{\mathcal{R}} = &(v_9v_{10}-v_2v_5-v_1v_6, 2v_8v_{10}-v_4v_5-v_3v_6, v_7v_{10}+v_2v_3-v_1v_4, 2v_2v_{10}+v_7v_4-v_9v_6, \\
&2v_1v_{10}-v_7v_3-v_9v_5,v_9^2-v_{10}^2+v_3v_4,v_8v_9-v_2v_3-v_1v_4, 2v_7v_9-v_4v_5+v_3v_6,\\
&2v_2v_9+v_8v_4-v_{10}v_6,2v_1v_9+v_8v_3-v_{10}v_5,v_8^2+v_{10}^2-v_3v_4-v_5v_6, v_7v_8+v_2v_5-v_1v_6,\\
&2v_2v_8-v_9v_4+v_7v_6,2v_1v_8-v_9v_3-v_7v_5,
v_7^2+v_{10}^2-v_5v_6, 2v_2v_7-v_{10}v_4+v_8v_6,\\
&2v_1v_7+v_{10}v_3-v_8v_5, 4v_2^2+v_4^2-v_6^2, 4v_1v_2-2v_{10}^2+v_3v_4+v_5v_6,\\
&4v_1^2+v_3^2-v_5^2, 2v_{10}^3-2v_{10}v_3v_4+v_8v_4v_5+v_8v_3v_6-2v_{10}v_5v_6).
%\end{split}
\end{align*}

Note that the degree of the first 10 generators of $\mathcal{R}(W)$ in variables $x,y,z,w$ is~2, the same as their degree in variables $t_1, t_2$, hence it is not surprising that the elements of $I_{\mathcal{R}}$ do not contain $v_{11}, v_{12}$.

Now we proceed with determining the singularities of the quotient of $Z = V(I_{\mathcal{R}}) \subseteq \CC^{12}$ by the action of the Picard quasitorus $T\times K$. First, the weight matrix $D$ in Proposition~\ref{prop:generators_R(W)} shows that the subgroup generated by $(-1,-1) \in T$ acts trivially on $Z$,
so we consider the action of $T' := T/\langle (-1,-1) \rangle$,
that has trivial isotropy groups on an open subset of $Z$.
To find a suitable linearisation $\chi$,  we compute the common refinement of the system of orbit cones in the character lattice of $T'$ and obtain a single chamber in the movable cone of $Z$. We pick a linearisation $(1,1)$ from the interior of the chamber and compute the corresponding set of~92 semistable $(\CC^*)^{12}$ orbits in $\CC^{12}$ which intersect $Z$ nontrivially. We present the first 23 of them in the form of lists of indices of variables non-vanishing on each orbit. The remaining 69 orbits are obtained from these lists by extending each one either by $11$ or by $12$ or by $11,12$.

\begin{align*}
   &(1,2,3,4,5,6,7,8,9,10),\quad (2,3,4,5,6,7,8,9,10),\quad (1,3,4,5,6,7,8,9,10),\\
& (1,2,4,5,6,7,8,9,10),\quad (1,2,3,5,6,7,8,9,10),\quad (1,2,3,4,6,7,8,9,10),\\
&  (4,6,7,8,9,10),\quad (1,2,3,4,5,7,8,9,10),\quad (3,5,7,8,9,10),\quad (1,2,3,6,7,8),\\
& (1,2,3,4,5,6,8,9,10),\quad (1,3,4,5,6,8,9,10),\quad (1,2,3,4,5,6,7,9,10),\\
& (2,3,4,5,6,7,9,10),\quad (1,2,5,6,9,10),\quad (1,2,3,4,9,10),\quad (1,2,4,5,7,8),\\
& (1,2,3,4,5,6,7,8,10),\quad (1,3,4,5,6,7,8,10),\quad (1,2,3,4,5,6,7,10),\\
& (1,2,3,4,5,6,7,8,9),\quad (2,3,4,5,6,7,8,9),\quad (1,2,3,4,5,6,8,9).
\end{align*}

It turns out that the semistable subset $Z_{\chi}^{ss}$ is smooth. However, the direct approach to prove it would involve computing the singular locus of $Z$, and the complexity of this computation is too high. Thus we follow the idea presented in~\cite[Section 4.C]{DW} and look for monomials vanishing on the set of singular points of $Z$, coming from minors of the Jacobian matrix.
This leads to proving the smoothness of $Z_{\chi}^{ss}$ which, together with the freeness of the action of $T'$, implies the smoothness of $Z_{\chi}^{ss}/T'$.

The last step is to deal with the action of $K$ on the obtained GIT quotient and find if it produces any (quotient) singularities. This can be done directly by investigating the action of $K$ on $\CC^{12}$ and looking for non-trivial isotropy groups on each $(\CC^*)^{12}$ orbit intersecting $Z$ separately.

\begin{prop}
The terminalisation of $V/G$ for $G = \mathbb{Z}_2\times Q_8$ has six singular points of type $\frac{1}{2} (1,1,1,1)$.
\end{prop}

\begin{proof}
The action of $T' \times K$ on $\CC^{12} \supset Z$ is given by the weight matrix $D$, and to describe the action of $K \simeq G/H\simeq \ZZ/2\ZZ$ we choose representatives of two elements in $G/H$ and look at their action on the generator of $\mathcal{R}(X)$. Thus, we consider the action on $\CC^{12}$ given by
\begin{align*}
(x_1,\ldots, x_{12}) &\mapsto\\ 
&((-1)^as_1^2x_1, (-1)^as_2^2x_2, (-1)^as_1s_2x_3, (-1)^bs_1s_2x_4, s_1s_2x_5, (-1)^{a+b}s_1s_2x_6,\\
&(-1)^{a+b}s_1^2x_7, (-1)^{a+b}s_2^2x_8, (-1)^bs_1^2x_9, (-1)^bs_2^2x_{10}, s_1^{-2}x_{11}, s_2^{-2}x_{12}),
\end{align*}
where $a,b\in\{0,1\}$ correspond to generators of $K$ and $s_1, s_2$ to coordinates of $T'$.

We already know that taking $a=b=0$ does not give non-trivial isotropy groups of the $T'\times K$-action. Consider the case $a=1, b = 0$ and check when a point of $Z^{ss}_{\chi}$ can have non-trivial isotropy. In this case, either $s_1s_2 = 1$ or $x_4 = x_5 = 0$. Assume first that $s_1s_2 = 1$. Then $x_3 = x_6 = 0$, and there are~4 such orbits: $O_{17} = (1,2,4,5,7,8)$ and its extensions by $11$ or $12$. On these orbits we have $s_1^2 = s_2^2 = -1$, hence $(s_1, s_2) \in \{(i,-i), (-i, i)\}$, which are equivalent up to $(-1,-1)$, i.e. they give the same point of $T'$. If such $(s_1, s_2)$ acts trivially, we need $x_{11} = x_{12} = 0$, so we obtain one orbit with  non-trivial isotropy group. One may check that the dimension of $O_{17} \cap Z$ is~2, hence its quotient by $T'\times K$ is a single point. Since the action of the isotropy group on the set of vanishing coordinates is the multiplication by $-1$, this is also the action on the tangent space to the quotient at $O_{17}/T'$, so the singularity type is $\frac{1}{2}(1,1,1,1)$.

If, in turn, $x_4 = x_5 = 0$ we see that those variables vanish simultaneously only on~4 orbits: $O_{10} = (1,2,3,6,7,8)$ and its extensions by $11$ or $12$. Then $-s_1^2 = -s_2^2 = -s_1s_2 = 1$, so $(s_1, s_2) \in \{ (i,i), (-i, -i)\}$, which are equivalent in $T'$. Again, if an orbit has this element in the isotropy group, then $x_{11}=x_{12} =0$, we obtain one such orbit, and one corresponding singular point of type $\frac{1}{2}(1,1,1,1)$ in the quotient.

The remaining cases ($a=1, b=0$ and $a=b=1$) are analogous and provide us with four more singular points of type $\frac{1}{2}(1,1,1,1)$ in the quotient, corresponding to orbits $O_7 = (4,6,7,8,9,10), O_9 = (3,5,7,8,9,10), O_{15} = (1,2,5,6,9,10), O_{16} = (1,2,3,4,9,10)$.
\end{proof}

\subsubsection{The group $D_{10}$}
While the terminalisation for the group $D_{10}$ has already been described in the literature (see Section~\ref{sec:terminalisation_D10}), it can be also constructed via Cox rings. Since the approach is very similar to the case of $\mathbb{Z}_2\times Q_8$ and the result already known, let us just summarize the results and mention the most interesting points.

The representation of $G \simeq D_{10}$ as in (\ref{eq:local_Dn}) has~5 conjugate symplectic reflections, generating~$G$. The remaining elements are in $[G,G] \simeq \mathbb{Z}_5$. The ring of $[G,G]$-invariants is generated by $15$ elements which, after homogenization and adding one more element corresponding to the junior class, generate the Cox ring of the terminalisation. It is worth noticing that, because there is just one junior class, the valuation lifting condition for generating the Cox ring is easier to check than in other cases.

Since we consider the action of $\mathbb{C}^*$, the movable cone and its chamber decomposition are simple: there is one chamber, hence one quotient. However, the ideal between the Cox rings generator is significantly bigger than for $\mathbb{Z}_2\times Q_8$. While the computation of semistable orbits and checking their stability work well (77 stable orbits), the direct smoothness check of the spectrum of the Cox ring does not finish in reasonable time. This is not surprising, as the same problem occured for smaller data for $\mathbb{Z}_2\times Q_8$, and one may use the same method again to determine singular orbits.

At the end, we encounter a single orbit consisting of singular points of the spectrum of the Cox ring, which gives one singular point in the quotient. Note that this time the singularity comes from the properties of the ideal, not from taking the quotient by a torsion part of the Picard quasitorus, which suggests that the singularity may be not a quotient one.

\subsection{Another approach to terminalise quotients by dihedral groups}\label{sec:terminalisation_D10}

Let $V$ be a vector space of dimension $2$ and $W_d$ denote the 2-dimensional representation of the dihedral group $D_{2d}$ generated by the matrices
\[\begin{bmatrix}
    0 & 1\\
    1 & 0
\end{bmatrix},\quad \begin{bmatrix}
    0 & \zeta\\
    \zeta^{-1} & 0
\end{bmatrix},\]
where $\zeta$ is a primitive $d$-root of unity. It is straightforward to check that the singularity $V\times V^{\ast}/W_d$ is the same as the one obtained in \eqref{eq:local_Dn}. 
It is shown in~\cite[Corollary 2.15]{BBFJLS} that the blow-up of the singular locus of the singularity $V\times V^{\ast}/W_d$ is locally analytically isomorphic to the Calogero-Moser space associated with $W_d$ at non-zero equal parameters \cite{Bonnafe2018}: this space admits an isolated singular point that is well described for small values of $d$. Moreover, note that this type of singularity has trivial local fundamental group \cite[Theorem 1.3]{BBFJLS}, so it is not a quotient singularity.

For $d=4$, i.e. for $D_8$, after the blow up of the singular locus the remaining isolated singularity is locally analytically isomorphic to the minimal nilpotent orbit closures of the Lie algebra $\mathfrak{sl}_3$.
This singularity is known to have a symplectic resolution with exceptional locus a $\mathbb P^2$ (indeed, this is the singularity occurring as the image of a small contraction on a \HK fourfold \cite{WW}).
This is coherent with the description of the minimal resolution of $\mathbb{C}^4/D_8$ described in~\cite[Proposition 5.22]{DG}, where it is proven that the exceptional locus above $0$ admits two surface components, one of them isomorphic to $\mathbb{P}^2$, the other to the Hirzebruch surface $\mathbb F_4$. 

In the case $d=5$, i.e. for $D_{10}$, after the blow up of the singular locus the remaining isolated singularity does not admit a symplectic resolution (if so, it would be a small symplectic contraction, but these are known in dimension $4$). By~\cite[Theorem 1.3]{BBFJLS}, this singularity is not locally analytically isomorphic to a singularity of a minimal orbit closure of a Lie algebra.
It is however, as described in~\cite[Section 12.3]{FJLS}, the nilpotent orbit closure of $A_4 +A_3$ along the codimension $4$ boundary $A_4 +A_2 +A_1$ in $E_8$ (following the same notation from Loc. Cit., where an explicit description is given).
Moreover, it follows from~\cite{Bel16} that the singularity is $\QQ$-factorial terminal (see~\cite[Section 5.2]{BBFJLS}).

\section{Proof of the main theorem \ref{main}} \label{sec:proof_of_main}

The $\QQ$-factorial terminalisation $\tilde X$ of $X/L_3(4)$ is an ISV with simply connected smooth locus by~\cite[Proposition 1.1]{Menet33}, as $L_3(4)$ is generated by involutions.

To compute $b_2(\tilde X)$ we apply \cite[Theorem 1.6]{BGMM}. The singular locus of $X/L_3(4)$ has only one component in codimension $2$, a surface $S$, and the general point of this surface is an ordinary double point for $X/L_3(4)$; indeed there is one conjugacy class of involutions, and no order three elements fixing a surface (see the Table at he end). Thus we have $b_2(\tilde X)=\mathrm{rk} H^2(X,\ZZ)^{L_3(4)}+1=4$. 

In order to find the singularities of $\tilde X$ we use Theorem \ref{main2}, and the results about $\QQ$-factorial terminalisations from Section \ref{sec:terminalisations}. Indeed, some of the singularities of $X/L_3(4)$ are already terminal (see Proposition \ref{prop:isolated_sings}), while some lie on the surface $S$ (see Proposition \ref{prop:non_isolated_sings}). Thanks to the local nature of the $\QQ$-factorial terminalisation \cite[Section 3]{BCHM2010}, we can study separately each singularity appearing at special points of $S$. After the terminalisation process, we are thus left with new isolated singularites: the singularities $\CC^4/S_3$ and $\CC^4/D_8$ can be resolved, but from points locally of the form $\CC^4/\ZZ_4, \CC^4/Q_8\times\ZZ_2$ we obtain 16 points of the form $\frac{1}{2}(1,1,1,1)$, while from points locally of the form $\CC^4/D_{10}$ we obtain two nonquotient singularities (see Section \ref{sec:terminalisation_D10}). 

\appendix
\section{The fixed K3 surface}\label{Ap:K3}
Let $F \subset X$ be one of the K3 surfaces fixed on $X$ by an involution of $L_3(4)$. Then, $F$ is the double cover of a quadric surface in $\mathbb P^3$ given by a (not very ample) polarization of degree 4; moreover, it admits a symplectic action of $Q_8\ast Q_8$ (that is the normaliser group of any involution in $L_3(4)$), and is a generalised Nikulin surface. The latter are a deformation family of dimension 13 of K3 surfaces, whose Néron-Severi lattice admits a primitive embedding of the lattice $E_7(-2)$: any K3 surface fixed by a symplectic involution on a K3$^{[2]}$-type manifold belongs to this family. We refer the reader to the upcoming paper \cite{newpaper} for more details about generalized Nikulin surfaces.
\begin{prop}\label{prop:fixedK3}
    The surface $F$ is isomorphic to the unique K3 surface with transcendental lattice $T(F)=\begin{bmatrix}
    20 &8\\8&20
\end{bmatrix}$.
\end{prop}
\begin{proof}
In~\cite{newpaper} it is proven that given a symplectic involution $\iota$ on a K3$^{[2]}$-type manifold $X$, if $F$ is the K3 surface fixed by $\iota$ and $N$ is the Nikulin orbifold obtained as terminalisation of $X/\iota$, then it holds $T(N)\simeq T(F)$. Therefore, since in our case $X$ is of maximal Picard rank, also $F$ is; thus, $F$ is also unique up to isomorphisms.

Recall that the invariant sublattice of $H^2(X,\mathbb Z)$ for the action of $\iota$ is isomorphic to $E_8(-2)\oplus U^{\oplus 3}\oplus \langle -2\rangle$, while the anti-invariant lattice (its orthogonal complement) is isomorphic to $E_8(-2)$; recall also that, since $H^2(X,\mathbb Z)$ is unimodular, up to a choice of marking there are integer elements of the form $(\varepsilon+\omega)/2$, with $\varepsilon\in E_8(-2)$ invariant, and $\omega\in E_8(-2)$ anti-invariant. Any invariant element $x\in H^2(X,\ZZ)$ can be written as $x=a\varepsilon+b\upsilon+c\delta$, with $\varepsilon\in E_8(-2),\ \upsilon\in U^{\oplus 3}$ and $\delta$ the generator of $\langle -2\rangle$.

The rational quotient map $\pi\colon X\to N$ induces a map $H^2(X,\mathbb Z)\rightarrow H^2(N,\mathbb Z)$, and by push-pull we can compute $(\pi_*x)^2=2x^2$; however, for any invariant $\varepsilon\in E_8(-2)$, its image $\pi_*\varepsilon\in H^2(N,\mathbb Z)$ is not primitive: it can be divided by 2, and $(\pi_*\varepsilon/2)^2=\varepsilon^2/2$. 

Now, since $\iota$ is symplectic, the transcendental lattice $T(X)$ is invariant; moreover, $T(N)$ is primitive by definition. Therefore, in order to obtain $T(N)=T(X)(2)$ we prove that, in our case, for any $x\in T(X)$ it holds $a=0$. 

But $a\neq 0$ if and only if there exists an anti-invariant element $\omega\in E_8(-2)$ such that $(x+\omega)/2$ is integer in $H^2(X,\mathbb Z)$, and we can verify this condition by computer (using~\cite{Sage}). Indeed, denote $I$, $\Omega$ respectively the invariant and co-invariant lattice for the action of $L_3(4)$ on $H^2(X,\ZZ)$: we fix the embeddings $E_8(-2)\hookrightarrow\Omega\hookrightarrow H^2(X,\ZZ)$ in such a way that $E_8(-2)$ is the anti-invariant lattice for one of the involutions in $L_3(4)$, and $I=\Omega^\perp$ is isomorphic to the lattice in \eqref{eq:TX}; then, we prove that $E_8(-2)\oplus I$ is primitive in $H^2(X,\ZZ)$ to conclude the proof.
\end{proof}

\section{Other Groups}\label{Ap:groups}

In this section, we focus on other large finite groups admitting a symplectic action on K3$^{[2]}$-type fourfolds:
$A_7$, $L_2(11)$, $M_{10}$, $\mathbb Z_2\times L_2(7)$ and $\mathbb{Z}_2^4:S_5$. In all these cases, the action on the second integral cohomology lattice has co-invariant lattice of rank 20, and the choice of an invariant polarisation of Beauville-Bogomolov degree 2 gives a rigid projective model. 
\begin{rem}\label{rem:proj_models_others}\textit{Projective models.}
We remark that the action of $L_2(11)$ considered in~\cite{Mazzon} is different from the one we consider: indeed, while they share the same co-invariant lattice, the invariant lattice that contains a class of square 2 (giving a double EPW model) and the one that contain a class of square 6 and divisibility 2 (giving a Fano model) are not the same \cite{Wawak}. For the group $G=A_7$ two non-isomorphic actions on double EPW models exist \cite{BilliWawak}, again distinguished by their invariant lattices; one of these two actions also admits a Fano model, that is the one considered in~\cite{Mazzon}.

For $\mathbb Z_2\times L_2(7)$ and $\mathbb{Z}_2^4:S_5$ the double EPW model is degenerate, and $\pi\colon X\to Y$ is a 6:1 cover of quadric hypersurface: these examples are birational to Hilbert squares of quartic K3 surfaces $S$ with no lines, and the symplectic action on $S^{[2]}$ of $G$ is generated with a contribution of Beauville's involution.
\end{rem}
 
\begin{thm}\label{main1}
Let $G$ be one of the groups $A_7$, $L_2(11)$, $\mathbb Z_2\times L_2(7)$, $\mathbb{Z}_2^4:S_5$ and $M_{10}$. Let $X$ be a K3$^{[2]}$-type manifold admitting a symplectic action of $G$ and an invariant polarisation of degree 2. Let $\widetilde{X/G}$ be the terminalisation of the quotient $X/G$. Then:
    \begin{enumerate}
        \item if $G\in\{A_7, L_2(11)\}$, then $\widetilde{X/G}$ is an ISV with simply connected smooth locus, non-quotient singularities and $b_2(\widetilde{X/G})=4$;
        \item if $G=M_{10}$, then $\widetilde{X/G}$ is a ISV with $\pi_1(\widetilde{X/G}^{sm})=\mathbb Z_2$ and $b_2(\widetilde{X/G})=4$.
        \item if $G=\mathbb Z_2\times A_2(7)$, then $\widetilde{X/G}$ is an ISV with simply connected smooth locus, and $b_2(\widetilde{X/G})=6$;
        \item if $G=\mathbb{Z}_2^4:S_5$, then $\widetilde{X/G}$ is an ISV with simply connected smooth locus, and $b_2(\widetilde{X/G})=8$.

    \end{enumerate}
\end{thm}
\begin{proof}
The computation of $b_2(\tilde X)$ follows from~\cite[Proposition 6.1]{BGMM} similarly to the proof of Theorem \ref{main}, that is
\begin{align*}
    b_2(\tilde X) = \text{rk}({H}^2(X, \mathbb{Z})) + N_2=3+N_2,
\end{align*}
where $N_2$ is the number of conjugacy classes of involutions in $G$: indeed, in all the cases involutions are the only morphisms with fixed locus in codimension 2. We can verify this directly on the double EPW model for cases $G =A_7, L_2(11)$ and $M_{10}$ (see Appendix \ref{Ap:fix}), while for cases $G = \mathbb Z_2\times L_2(7)$ and $\mathbb{Z}_2^4:S_5$ it follows from the fact all order 3 automorphisms are natural: indeed, since for the latter two actions the projective models are birational to Hilbert squares of K3 surfaces with Beauville's involution, elements of order 3 in $G$ are necessarily natural, and therefore fix points on $X$. The computation of $\pi_1(\tilde X^{sm})$ follows from~\cite[Proposition 8.1]{BGMM}, that is
\begin{align*}
    \pi_1(\tilde X^{sm}) \cong G/N
\end{align*}
where $N$ is the subgroup of $G$ generated by elements with codimension 2 fixed loci, and from the fact that all the groups we consider (except for $M_{10}$) are generated by involutions. For $M_{10}$, the subgroup generated by involution has index 2 (it is isomorphic to $A_6$).

Since the actions of $G\in\{A_7,L_2(11)\}$ on $\PP^5$ admit $D_{10}$ as subgroup fixing singular points on the EPW sextic $Y\subset\PP^5$ (see Appendix \ref{Ap:fix}), by~\cite{BBFJLS}  we can conclude that the ISV obtained as a terminalisation of $X/G$ will have non-quotient singularities (see also Section \ref{sec:terminalisation_D10}): indeed, the stabiliser of $y\in Y^{sing}$ is the same as that of $\pi^{-1}(y)=x\in X$.
\end{proof}

Natural problems arise.

\begin{prob}
    Are the ISVs obtained in Theorem \ref{main1} from the terminalisation of $X/A_7$ and $X/L_3(4)$ deformation equivalent?
\end{prob}
We remark that, while $A_7$ contains elements of order 6, $\ZZ_6$ does not fix isolated points on the double EPW sextic $X$ (all its fixed points lie on K3 surfaces fixed by involutions). However, the action of $A_7$ on the EPW sextic $Y$ produces a singular point $q$ fixed by the group $H$ with GAP Id (24,8); $\pi^{-1}q$ is a point, also fixed by $H$. This probably distinguishes this case from the $L_3(4)$ case, but we did not compute the terminalisation of $\CC^4/H$; we remark that $H$ is generated by junior elements, it contains $D_{12}$ as normal subgroup, and $\CC^4/D_{12}$ also has a non-quotient terminalisation \cite{BBFJLS}.

\begin{prob}
    Consider the two ISVs obtained from the terminalisation of $X/L_2(11)$, namely $Z_1$ when $X$ is a double EPW sextic as in Theorem \ref{main1}, and $Z_2$ when $X$ is a Fano variety of lines of a cubic fourfold as in~\cite{Mazzon}. Are $Z_1$ and $Z_2$ deformation equivalent?
    The same question applies to the two ISVs obtained from the group $M_{10}$, and the three obtained from the group $A_7$ (two from double EPWs, one from a Fano).
\end{prob}
For $L_2(11)$ and $M_{10}$, since the two models are obtained from actions with different invariant lattices, we expect a negative answer. On the other hand, for $A_7$ the two double EPW models (obtained from two different invariant lattices) behave remarkably alike in terms of points with non-trivial stabiliser, and the Fano model belongs to the same deformation family as one of the two double EPW models: therefore, in this case we expect a positive answer. 

\begin{prob}
    What is the terminalisation of the singularity $\CC^4/A_5$?
\end{prob}
The EPW sextic $Y$ invariant for an action of $M_{10}$ has a singular point fixed by the group $A_5$: therefore, this group will fix also its preimage on the double EPW sextic $X\rightarrow Y$. This is the biggest group action fixing points on a double EPW sextic we found in our analysis. One might be interested in studying the (hopefully rather pathological) corresponding residual singularity, appearing in the terminalisation of $X/M_{10}$.

\begin{prob}
    Perform an analogous analysis in the case of sixfolds.
\end{prob}
In fact, in~\cite{BMW} a list of symplectic actions on EPW cubes is provided. This gives rise to a list of examples of ISV sixfolds that can be studied with the methods from this paper. 

\section{Codes}\label{Ap:codes}
In this section we present the code to construct the EPW sextic $Y\subset\PP^5$ symmetric with respect to the action of $L_3(4)$. The projective representation of $L_3(4)$ comes from~\cite{ATLAS}.

\subsection{The symmetric Lagrangian subspace}
 We start with the computation of a Lagrangian subspace $A\subset\bigwedge^3\CC^6$ symmetric for the action of $L_3(4)$, using~\cite{GAP}.
\begin{lstlisting}[language=GAP]
LoadPackage("orb");
InducedMapOnWedge := function(Mat)
#computes the matrix of the linear map induced by the matrix Mat on 
#the third Exterior Power of the underlying space in the basis of 
#lexographically ordered simple vectors 
#obtainted through multiplication of the canonical basis
	#...
end;;

CheckLagrangianWedge3_6 := function(L) 
#checks if a subspace of Wedge^3V_6 is a Lagrangian space
	#...
end;;

OrbitSpace := function(vec, Gr, F)
#returns the space spanned by the orbit of 
#the vector vec by the group Gr over the field F
	#...
end;;


b := E(7)+E(7)^2+E(7)^4; B := -1-b;  # b7, b7**
Fie := CF(420);
#L_3(4) generators per [ATLAS]
g1 := [
[-1,0,0,0,0,0],
[0,0,1,0,0,0],
[0,1,0,0,0,0],
[0,0,0,0,1,0],
[0,0,0,1,0,0],
[0,0,0,0,0,-1]];
g2 := [
[0,1,0,0,0,0],
[0,-1,0,0,1,E(3)+1],
[0,0,-1,0,0,0],
[0,0,-E(3)-1,0,0,1],
[-1,-1,0,E(3)+1,0,0],
[0,0,0,-1,0,0]];
6L3_4 = Group(g1, g2);



#######LAGRANGIAN SPACES#######		
L3_4_20 := Group([InducedMapOnWedge(g1), InducedMapOnWedge(g2)]);
CC := ConjugacyClasses(L3_4_20);
ImportantCC := [];
for C in CC do
	if Trace(Representative(C)) > 0 then
		Add(ImportantCC, C);
	fi;
od;
flag := true;
for ind in [1..Length(ImportantCC)] do
	if Trace(Representative(ImportantCC[ind])) = 1 then
		if flag then
			ind1 := ind;
			flag := false;
		else
			ind2 := ind;
		fi;
	fi;
od;

e1 := CanonicalBasis(Rationals^20)[1];
#projector onto subspace invariant under the a subreprepresentation
P1 := 0 * IdentityMat(20); 
for ind in [1..Length(ImportantCC)] do
	Paux := 0 * IdentityMat(20);
	for g in ImportantCC[ind] do
		Paux := Paux + g;
	od;
	if ind = ind1 then
		P1 := P1 - b * Paux;
	elif ind = ind2 then
		P1 := P1 - B * Paux;
	else 
		P1 := P1 + Trace(Representative(ImportantCC[ind]))/2 * Paux;
	fi;
od;
v1 := e1 * P1;;
V1 := OrbitSpace(v1, L3_4_20, Fie);;
CheckLagrangianWedge3_6(V1);
Basis1 := Basis(V1);;
\end{lstlisting}

\subsection{The EPW sextic and the action of $L_3(4)$ on $\PP^5$}
What follows is \cite{M2} code that constructs $Y\subset\PP^5$ using the Lagrangian space from above.

\begin{lstlisting}[language=Macaulay2]
--computes the symmetrix matrix for the Lagrangian in appropriate basis
symmetricMatrix = BasisOfLagrangian -> ( 
--BasisOfLagrangian given as a list of lists
	BasisOfDomain = {};
	BasisOfCodomain = {};
	for v in BasisOfLagrangian do (
		v1 = take (v, {0,9}); --the first 10 coordinates
		v2t = take (v, {10,19}); --the last 10 coordinates
		--reordering the last 10 coordinates to get it in the basis 
    --dual to the one in which v1 is written 
		mult = {1,-1,1,-1,1,-1,1,1,-1,1};
		v2 = {};
		for i from 0 to 9 do (
			v2 = append(v2, mult_i * v2t_(9-i));
		);
			
		--resulting halves of vectors added to the bases
    --of the domain/codomain space
		BasisOfDomain = append(BasisOfDomain, v1);
		BasisOfCodomain = append(BasisOfCodomain, v2);	
	);
		
	--change of bases to the canonical ones
    M1 = inverse transpose matrix BasisOfDomain;
	M2 = transpose matrix BasisOfCodomain;
	M2 * M1
);

p = 65521; 
F = ZZ/p;
R = F[t];
P = F[x,y,z,u,v,w];
t = -18153_F; --root of unity of order 420
ww = t^20;

Basis1 = --basis obtained from the GAP code above
Mat1 = symmetricMatrix(Basis1);
M=matrix{
	{ 0, 0,0,0,0,0,0,w,-v,u},
	{0,0,0,0,0,-w,v,0,0,-z},
	{ 0,0,0,0,w,0,-u,0,z,0},
	{0,0,0,0,-v,u,0,-z,0,0},
	{ 0,0,0,0,0,0,0,0,0,y},
	{ 0,0,0,0,0,0,0,0,-y,0},
	{ 0,0,0,0,0,0,0,y,0,0},
	{ 0,0,0,0,0,0,0,0,0,0},
	{ 0,0,0,0,0,0,0,0,0,0},
	{ 0,0,0,0,0,0,0,0,0,0}
};
MM = M + transpose ( M );
Lambda1 = Mat1-MM;	
	
d1 = det(Lambda1); --the equation for the EPW
I1 = ideal d1;
degree(I1) --returns 6
s1 = ideal singularLocus I1; --singular locus of the EPW
dim(s1) --returns 3, so the projective dim is 2
degree(s1) --returns 40
I1h = homogenize(I1,x) --ideal of EPW sextic Y
S1h = ideal singularLocus I1h

--action of L_3(4):
--involution
g1 = 1_F * (matrix {
{-1,0,0,0,0,0},
{0,0,1,0,0,0},
{0,1,0,0,0,0},
{0,0,0,0,1,0},
{0,0,0,1,0,0},
{0,0,0,0,0,-1}});
--order 5 generator
g2 = 1_F * (matrix {
{0,1,0,0,0,0},
{0,-1,0,0,1,ww^7+1},
{0,0,-1,0,0,0},
{0,0,-ww^7-1,0,0,1},
{-1,-1,0,ww^7+1,0,0},
{0,0,0,-1,0,0}});
g1 = transpose g1
g2 = transpose g2
mapg1 = map(P, P, g1)
mapg2 = map(P, P, g2)
\end{lstlisting}

\section{Fixed points for maximal group actions on $\PP^5$}\label{Ap:fix}
We collect here tables about the action of a group $G\in\{L_3(4), A_7, L_2(11), M_{10}\}$ on a $G$-invariant EPW sextic $Y\subset\PP^5$.

For each $G$ we list the conjugacy classes of subgroups that fix some non-empty subset of $Y$: we remark that there may be subgroups of $G$ such that none of the projective subspaces invariant for their action on $\PP^5$ intersect $Y$; these subgroups are therefore excluded from the tables.

For each subgroup, the column ``Fix $\PP^5$'' collects the dimensions of the invariant subspaces for its action on $\PP^5$; the columns ``Fix $Y$'',``Fix $Y^{sing}$'' describe the intersection of said subspaces with $Y$ and its singular locus, and here $C_n,S_n$ denote a curve (resp. a surface) of degree $n$.\\

\footnotesize
%\begin{center}
\begin{table}[!h]
\begin{tabular}{|c|c|c|c|c|c|}
        \hline \hline
        \# & Group & Gap ID & Fix $\mathbb{P}^5$ & Fix $Y$ & Fix $Y^{sing}$\\ \hline \hline 
    1 & $1$ &  1, 1  & 5 & Y & $Y^{sing}$ \\ \hline
    2 & $\mathbb{Z}_2$ &  2, 1  & 3 1 & $S_2 \cup S_4\cup \text{6 pt}$ & $C_8$\\ \hline
    3 & $\mathbb{Z}_3$ &  3, 1  & 1 1 1 & 15 pt & 3 pt \\ \hline
    4 & $\mathbb{Z}_2 \times \mathbb{Z}_2$ &  4, 2  & 2 0 0 0 & $3 C_2$ & 8 pt\\ \hline
    5 & $\mathbb{Z}_5$ &  5, 1  & 1 0 0 0 0 & 8 pt & 2 pt \\ \hline
    6 & $S_3$ &  6, 1  & 1 & 5 pt & 1 pt \\ \hline
    7 & $S_3$ &  6, 1  & 1 & 5 pt & 1 pt\\ \hline
    8 & $\mathbb{Z}_6$ &  6, 2  & 1 0 0 0 0 & 7 pt & 3 pt\\ \hline
    9 & $D_{10}$ &  10, 1  & 1 & 4 pt & 2 pt \\ \hline
    10 & $\mathbb{Z}_{11}$ &  11, 1  & 0 0 0 0 0 0 & 5 pt & 5 pt\\ \hline
    %11 & $A_4$ &  12, 3  & 0 0 0 & - & -\\ \hline
    11 & $D_{12}$ &  12, 4  & 1 & 5 pt & 1pt\\ \hline
   % 13 & $\mathbb{Z}_{11} : \mathbb{Z}_5$ &  55, 1  & 0 & - & - \\ \hline
  %  14 & $A_5$ &  60, 5  & 0 & - & -\\ \hline
 %   15 & $A_5$ &  60, 5  & 0 & - & -\\ \hline
%    16 & $L_2(11)$ &  660, 13  & 0 & - & - \\ \hline
\end{tabular}
    \caption{Fixed points for the action of $L_2(11)$}
\end{table}

%\end{center}
\normalsize
\footnotesize
%\begin{center}

\begin{table}[!h]
\begin{tabular}{|c|c|c|c|c|c|}
        \hline \hline
        \# & Group & Gap ID & Fix $\mathbb{P}^5$ & Fix $Y$ & Fix $Y^{sing}$\\ \hline \hline 
        1 & $1$ &  1, 1  & 5 & Y & $Y^{sing}$ \\ \hline
        2 & $\mathbb{Z}_2$ &  2, 1  & 3 1 & $S_2 \cup S_4\cup \text{6 pt}$ & $C_8$\\ \hline
        3 & $\mathbb{Z}_3$ &  3, 1  & 1 1 1 & 13 pt & 1 pt \\ \hline
        4 & $\mathbb{Z}_2 \times \mathbb{Z}_2$ &  4, 2  & 2 0 0 0 & $6C_1$ & 15 pt \\ \hline
        5 & $\mathbb{Z}_4$ &  4, 1  & 1 1 0 0 & 12 pt & - \\ \hline
        6 & $\mathbb{Z}_5$ &  5, 1  & 1 0 0 0 0 & 13 pt & 1 pt \\ \hline
        7 & $S_3$ &  6, 1  & 1 & 1 pt & 1 pt \\ \hline
        %8 & $\mathbb{Z}_8$ &  8, 1  & 0 0 \\ \hline
        8 & $D_8$ &  8, 3  & 1 0 0 & 6 pt & - \\ \hline
        9 & $D_{10}$ &  10, 1  & 1 & 1 pt & 1 pt\\ \hline
        10 & $A_4$ &  12, 3  & 0 0 0 & 1 pt & 1 pt\\ \hline
        %12 & $S_4$ &  24, 12  & 0 \\ \hline
        11 & $A_5$ &  60, 5  & 0 & 1 pt & 1 pt\\ \hline
\end{tabular}    
    \caption{Fixed points for the action of $M_{10}$}
\end{table}

%\end{center}
\normalsize
\footnotesize
%\begin{center}
\begin{table}[!h]
\begin{tabular}{|c|c|c|c|c|c|}
        \hline \hline
        \# & Group & Gap ID & Fix $\mathbb{P}^5$ & Fix $Y$ & Fix $Y^{sing}$\\ \hline \hline 
        1 & $1$ &  1, 1  & 5 & Y & $Y^{sing}$ \\ \hline
        2 & $\mathbb{Z}_2$ &  2, 1  & 3 1 & $S_2 \cup S_4\cup \text{6 pt}$ & $C_8$\\ \hline
        3 & $\mathbb{Z}_3$ &  3, 1  & 1 1 1 & 15 pt & 3 pt \\ \hline
		4 & $\mathbb{Z}_3$ &  3, 1  & 1 1 1 & 15 pt & 3 pt \\ \hline
		5 & $\mathbb{Z}_2 \times \mathbb{Z}_2$ &  4, 2  & 2 0 0 0 & $3C_2$ & 8pt \\ \hline
		6 & $\mathbb{Z}_2 \times \mathbb{Z}_2$ &  4, 2  & 2 0 0 0 & $3C_2$ & 8pt \\ \hline
		7 & $\mathbb{Z}_4$ &  4, 1  & 1 1 0 0 & 10pt & 4pt \\ \hline
		8 & $\mathbb{Z}_5$ &  5, 1  & 1 0 0 0 0 & 8pt & 2pt\\ \hline
		9 & $\mathbb{Z}_6$ &  6, 2  & 1 0 0 0 0 & 7pt & 3pt \\ \hline
		10 & $S_3$ &  6, 1  & 1 & 5pt & 1pt\\ \hline
		11 & $S_3$ &  6, 1  & 1 & 5pt & 1pt\\ \hline
		12 & $\mathbb{Z}_7$ &  7, 1  & 0 0 0 0 0 0 & 6pt & 3pt \\ \hline
		13 & $D_8$ &  8, 3  & 1 0 0 & 4pt & 2pt \\ \hline
        14 & $D_{10}$ &  10, 1  & 1 & 4pt & 2pt\\ \hline
        %15 & $A_4$ &  12, 3  & 0 0 0 \\ \hline
		%16 & $A_4$ &  12, 3  & 0 0 0 \\ \hline
        %17 & $A_4$ &  12, 3  & 0 0 0\\ \hline
        %18 & $A_4$ &  12, 3  & 0 0 0 \\ \hline
		15 & $\mathbb{Z}_6 \times \mathbb{Z}_2$ &  12, 5  & 0 0 0 0 0 0 & 3pt & 3pt\\ \hline
		16 & $\mathbb{Z}_3 : \mathbb{Z}_4$ &  12, 1  & 0 0 & 1pt & 1pt  \\ \hline
		  17 & $D_{12}$ &  12, 4  & 1 & 5pt & 1pt \\ \hline
		%22 & $\mathbb{Z}_5 : \mathbb{Z}_4$ &  20, 3  & 0 0 \\ \hline
		  18 & $(\mathbb{Z}_6 \times \mathbb{Z}_2) : \mathbb{Z}_2$ &  24, 8  & 0 0 & 1pt & 1pt \\ \hline
        %24 & $S_4$ &  24, 12  & 0\\ \hline
		%25 & $S_4$ &  24, 12  & 0 \\ \hline
		%26 & $S_4$ &  24, 12  & 0 \\ \hline
		%27 & $S_4$ &  24, 12  & 0 \\ \hline
		%28 & $A_5$ &  60, 5  & 0 \\ \hline
		%29 & $A_5$ &  60, 5  & 0 \\ \hline
		%30 & $S_5$ &  120, 34  & 0 \\ \hline
    
\end{tabular}
\caption{Fixed points for the action of $A_7$}
\end{table}
%\end{center}
\normalsize
\footnotesize
%\begin{center}
\begin{table}[!h]
\begin{tabular}{|c|c|c|c|c|c|}
        \hline \hline
        \# & Group & Gap ID & Fix $\mathbb{P}^5$ & Fix $Y$ & Fix $Y^{sing}$\\ \hline \hline 
1 & $1$ &  1, 1  & 5 & $Y$ & $Y^{sing}$ \\ \hline
2 & $\mathbb{Z}_2$ &  2, 1  & 3 1 & $S_2 \cup S_4\cup \text{6 pt}$ & $C_8$\\ \hline
3 & $\mathbb{Z}_3$ &  3, 1  & 1 1 1 & 15 pt & 3 pt \\ \hline
4 & $\mathbb{Z}_2 \times \mathbb{Z}_2$ &  4, 2  & 2 0 0 0 & $3C_2$ & 8 pt\\ \hline
5 & $\mathbb{Z}_2 \times \mathbb{Z}_2$ &  4, 2  & 2 0 0 0 & $3C_2$ & 8 pt \\ \hline
6 & $\mathbb{Z}_2 \times \mathbb{Z}_2$ &  4, 2  & 2 0 0 0 & $3C_2$ & 8 pt \\ \hline
7 & $\mathbb{Z}_2 \times \mathbb{Z}_2$ &  4, 2  & 2 0 0 0 & $3C_2$ & 8 pt  \\ \hline
8 & $\mathbb{Z}_2 \times \mathbb{Z}_2$ &  4, 2  & 1 1 1  & 18 pt & - \\ \hline
9 & $\mathbb{Z}_2 \times \mathbb{Z}_2$ &  4, 2  & 1 1 1 & 18 pt & - \\ \hline
10 & $\mathbb{Z}_2 \times \mathbb{Z}_2$ &  4, 2  & 1 1 1 & 18 pt & - \\ \hline
11 & $\mathbb{Z}_4$ &  4, 1  & 1 1 1 & $2C_1 \cup 6\text{pt}$ & 8 pt\\ \hline
12 & $\mathbb{Z}_4$ &  4, 1  & 1 1 0 0 & 10 pt & 4 pt\\ \hline
13 & $\mathbb{Z}_4$ &  4, 1  & 1 1 0 0 & 10 pt & 4 pt\\ \hline
14 & $\mathbb{Z}_5$ &  5, 1  & 1 0 0 0 0 & 8 pt & 2 pt\\ \hline
15 & $S_3$ &  6, 1  & 1 & 6 pt & 1 pt \\ \hline
16 & $\mathbb{Z}_7$ &  7, 1  & 0 0 0 0 0 0 & 6 pt & 3 pt \\ \hline
17 & $\mathbb{Z}_4 \times \mathbb{Z}_2$ &  8, 2  & 0 0 0 0 0 0  & 2 pt & -\\ \hline
18 & $\mathbb{Z}_4 \times \mathbb{Z}_2$ &  8, 2  & 1 0 0 0 0 & 10 pt & -\\ \hline
19 & $\mathbb{Z}_4 \times \mathbb{Z}_2$ &  8, 2  & 1 0 0 0 0 & 6 pt & -\\ \hline
20 & $\mathbb{Z}_2 \times \mathbb{Z}_2 \times \mathbb{Z}_2$ &  8, 5  & 1 0 0 0 0 & 2 pt & -\\ \hline
21 & $\mathbb{Z}_2 \times \mathbb{Z}_2 \times \mathbb{Z}_2$ &  8, 5  & 1 0 0 0 0  & 6 pt & -\\ \hline
22 & $D_8$ &  8, 3  & 1 0 0 & 4 pt & 2 pt\\ \hline
23 & $D_8$ &  8, 3  & 1 0 0 & 6 pt & -\\ \hline
24 & $D_8$ &  8, 3  & 1 & 4 pt & 2 pt\\ \hline
25 & $Q_8$ &  8, 4  & 0 0 & 2 pt & -\\ \hline
26 & $D_{10}$ &  10, 1  & 1 & 4 pt & 2 pt\\ \hline
27 & $\mathbb{Z}_4 \times \mathbb{Z}_4$ &  16, 2  & 0 0 0 0 0 0 & 6 pt & -\\ \hline
28 & $(\mathbb{Z}_4 \times \mathbb{Z}_2) : \mathbb{Z}_2$ &  16, 3  & 0 0 0 0 & 2 pt & - \\ \hline
29 & $(\mathbb{Z}_4 \times \mathbb{Z}_2) : \mathbb{Z}_2$ &  16, 3  & 0 0 0 0  & 2 pt & -\\ \hline
30 & $(\mathbb{Z}_4 \times \mathbb{Z}_2) : \mathbb{Z}_2$ &  16, 3  & 0 0 0 0 & 2 pt & -\\ \hline
31 & $(\mathbb{Z}_4 \times \mathbb{Z}_2) : \mathbb{Z}_2$ &  16, 3  & 0 0 & 2 pt & -\\ \hline
32 & $\mathbb{Z}_2 \times D_8$ &  16, 11  & 0 0 0 0 & 2 pt & -\\ \hline
33 & $\mathbb{Z}_2 \times D_8$ &  16, 11  & 1 & 6 pt & - \\ \hline
34 & $(\mathbb{Z}_4 \times \mathbb{Z}_4) : \mathbb{Z}_2$ &  32, 31  & 0 0 & 2 pt & - \\ \hline
\end{tabular}
    \caption{Fixed points for the action of $L_3(4)$}
\end{table}
%\end{center}
\normalsize
\clearpage

\bibliographystyle{halpha-abbrv}
\bibliography{references}

@article{Anno,
  author  = {Anno, R. E.},
  title   = {Four-{D}imensional {T}erminal {G}orenstein {Q}uotient {S}ingularities},
  journal = {Math. Notes},
  volume  = {73},
  number  = {5},
  year    = {2003},
  pages   = {769--776}
}

@article{LLX,
title = "Irreducible symplectic varieties with a large second {B}etti number",
author = "Yuchen Liu and Zhiyu Liu and Chenyang Xu",
year = "2025",
volume = "2025",
pages = "1--31",
journal = "J. Reine Angew. Math.",
number = "825",
}

@article{BBFJLS,
author = {Bellamy, Gwyn and Bonnafé, Cédric and Fu, Baohua and Juteau, Daniel and Levy, Paul and Sommers, Eric},
title = {A new family of isolated symplectic singularities with trivial local fundamental group},
journal = {Proc. London Math. Soc.},
volume = {126},
number = {5},
pages = {1496-1521},
year = {2023}
}

@article{BL,
  author  = {Bakker, B. and Lehn, Ch.},
  title   = {A global  {T}orelli theorem for singular symplectic varieties},
  journal = {J. Eur. Math. Soc.},
  volume  = {23},
  year    = {2021},
  pages    = {949–994}
}

@article{Beauville,
  title={Vari{\'e}t{\'e}s {K}{\"a}hleriennes dont la premi{\`e}re classe de {C}hern est nulle},
  author={Arnaud Beauville},
  journal={J. Differential Geom.},
  year={1983},
  volume={18},
  pages={755-782},
  }

@article{BGMM,
  author  = {Bertini, V. and Grossi, A. and Mauri, M. and Mazzon, E.},
  title   = {Terminalizations of quotients of compact hyperk{\"a}hler manifolds by induced symplectic automorphisms},
  journal = {\'Epijournal G\'eom. Alg\'ebrique},
  volume  = {9},
  year    = {2025},
  note    = {Article No. 14}
}

@article{Bel16,
author={Bellamy, Gwyn}, 
title={Counting resolutions of symplectic quotient singularities},
volume={152},
number={1},
journal={Compos. Math.}, 
year={2016}, 
pages={99–114}
}

@article{BilliWawak,
  author  = {Billi, S. and Wawak, T.},
  title   = {Double {EPW}-sextics with actions of $\mathcal{A}_7$ and irrational {GM} threefolds},
  journal = {Bull. Soc. Math. France},
  volume  = {152},
  number  = {4},
  year    = {2024},
  pages   = {857--868}
}

@article{BMW,
title = {On birational automorphisms of double {EPW}-cubes},
author = {Simone Billi and Stevell Muller and Tomasz Wawak},
year = {2025},
journal = {Math. Nachr.},
volume = {298},
number = {6},
pages = {1943–1963}
}

@incollection{Campana,
 author = {Campana, Fr{\'e}d{\'e}ric},
 title = {Orbifolds with trivial first {Chern} class},
 booktitle = {The Fano conference. Papers of the conference organized to commemorate the 50th anniversary of the death of Gino Fano (1871--1952), Torino, Italy, September 29--October 5, 2002},
 isbn = {88-900876-1-7},
 pages = {339--351},
 year = {2004},
 publisher = {Torino: Universit{\`a} di Torino, Dipartimento di Matematica},
}

@article{OGrady-EPW,
  title={Irreducible symplectic 4-folds and {E}isenbud-{P}opescu-{W}alter sextics},
  author = {{O}'{G}rady, Kieran},
  journal={Duke Math. J.},
  year={2005},
  volume={134},
  pages={99-137},
}

@article{HP,
    author = {H{\"o}ring, Andreas and Peternell, Thomas},
    title = {Algebraic integrability of foliations with numerically trivial canonical bundle},
    journal = {Invent. Math.},
    year = {2019},
    number = {2},
    pages = {395--419},
    volume = {216} 
}

@incollection{Perego,
 author = {Perego, Arvid},
 title = {Examples of {I}rreducible {S}ymplectic {V}arieties},
 booktitle = {Birational Geometry and Moduli Spaces},
 isbn = {978-3-030-37114-2},
 pages = {151--172},
 year = {2020},
 publisher = {Springer INdAM Series, vol. 39. Springer, Cham.},
}

@article{CGKK,
  author  = {Chiara Camere and Alice Garbagnati and Grzegorz Kapustka and Michal Kapustka},
  title   = {Projective orbifolds of {N}ikulin type},
  journal = {Algebra Number Theory},
  volume  = {18},
  number  = {1},
  year    = {2024},
  pages   = {165--208}
}

@unpublished{newpaper,
  author  = {Camere, C. and Garbagnati, A. and Kapustka, G. and Kapustka, M.},
  title   = {Hyper-K\"{a}hler fourfolds and generalised {N}ikulin surfaces},
  note ={(Upcoming).}
  }

@unpublished{DG,
  author  = {Donten-Bury, M. and Grab, M.},
  title   = {Cox rings of some symplectic resolutions of quotient singularities},
  note    = {(Preprint) arXiv:1504.07463}
}

@article{FJLS,
  author  = {Fu, B. and Juteau, D. and Levy, P. and Sommers, E.},
  title   = {Generic singularities of nilpotent orbit closures},
  journal = {Adv. Math.},
  volume  = {305},
  year    = {2017},
  pages   = {1--77}
}

@article{BayerPerry,
title = {Kuznetsov’s {F}ano threefold conjecture via {K3} categories and enhanced group actions},
author = {Arend Bayer and Alexander Perry},
pages = {107--153},
volume = {2023},
number = {800},
journal = {J. Reine Angew. Math.},
year = {2023},
}

@article{Bonnafe2018,
  author    = {Bonnaf{\'e}, C{\'e}dric},
  title     = {On the {C}alogero--{M}oser space associated with dihedral groups},
  journal   = {Ann. Math. Blaise Pascal},
  year      = {2018},
  volume    = {25},
  number    = {2},
  pages     = {265--298},
}

@incollection{Fujiki,
  author    = {Fujiki, A.},
  title     = {On primitively symplectic compact {K}{\"a}hler {V}-manifolds of dimension four},
  booktitle = {Classification of algebraic and analytic manifolds (Katata, 1982)},
  series    = {Progr. Math.},
  volume    = {39},
  publisher = {Boston : Birkhäuser},
  year      = {1983},
  pages     = {71--250}
}

@unpublished{Ghirlanda,
  author    = {Ghirlanda, Marco},
  title     = {A canonicity criterion for toric varieties and the classification of canonical 4-simplices},
  note   = {(Preprint) arXiv:2603.21198}
}

@article{HM,
  author  = {H{\"o}hn, G. and Mason, G.},
  title   = {Finite groups of symplectic automorphisms of hyperk{\"a}hler manifolds of type {K3}$^{[2]}$},
  journal = {Bull. Inst. Math. Acad. Sin. (N.S.)},
  volume  = {14},
  number  = {2},
  year    = {2014},
  pages   = {189--264}
}

@incollection {ItoReid,
    AUTHOR = {Ito, Yukari and Reid, Miles},
     TITLE = {The {M}c{K}ay correspondence for finite subgroups of {${\rm
              SL}(3,\bold C)$}},
 BOOKTITLE = {Higher-dimensional complex varieties ({T}rento, 1994)},
     PAGES = {221-240},
 PUBLISHER = {de Gruyter, Berlin},
      YEAR = {1996},
   MRCLASS = {14E15 (14F45 14L30)},
  MRNUMBER = {1463181 (98i:14018)},
MRREVIEWER = {Nicolas Pouyanne},
}

@article {KaledinMcKay,
    AUTHOR = {Kaledin, D.},
     TITLE = {Mc{K}ay correspondence for symplectic quotient singularities},
   JOURNAL = {Invent. Math.},
  FJOURNAL = {Inventiones Mathematicae},
    VOLUME = {148},
      YEAR = {2002},
    NUMBER = {1},
     PAGES = {151--175},
      ISSN = {0020-9910},
     CODEN = {INVMBH},
   MRCLASS = {14E15},
  MRNUMBER = {MR1892847 (2003d:14022)},
MRREVIEWER = {Nicolas Pouyanne},
}

@article{Men20,
  author  = {Menet, Gr{\'e}goire},
title   = {Global {T}orelli theorem for irreducible symplectic orbifolds},                                
journal = {J. Math. Pures Appl. (9)},
 year    = {2020},
 volume  = {137}, 
 pages   = {213-237} 
}

@unpublished{Menet33,
  author  = {Menet, Gr{\'e}goire},
  title   = {Thirty-three deformation classes of compact hyperk{\"a}hler orbifolds},
  note = {(Preprint) arXiv:2211.14524}
}

@phdthesis{MongardiThesis,
  author = {Mongardi, G.},
  title  = {Automorphisms of hyperk{\"a}hler manifolds},
  school = {Universit{\`a} degli Studi Roma Tre},
  year = {2013}
}

@article{Mongardi,
title = {Symplectic involutions on deformations of {K3}{$^{[2]}$}},
author = {Giovanni Mongardi},
pages = {1472--1485},
volume = {10},
number = {4},
journal = {Open Math.},
year = {2012},
}

@article{OGrady10,
author = {{O}'{G}rady, Kieran},
title = {Desingularized moduli spaces of sheaves on a {K3}},
pages = {49--117},
volume = {1999},
number = {512},
journal = {J. Reine Angew. Math.},
year = {1999},
}

@article{OGrady6,
author = {{O}'{G}rady, Kieran},
year = {2003},
pages = {435-505},
title = {A new six dimensional irreducible symplectic variety},
volume = {12},
journal = {J. Algebraic Geom.},
}

@unpublished{Wawak,
  author  = {Wawak, T.},
  title   = {Very symmetric hyper-{K}{\"a}hler fourfolds},
  note = {(Preprint) arXiv:2212.02900},
}

@article{Xiao,
  author  = {Xiao, G.},
  title   = {Galois covers between {K}3 surfaces},
  journal = {Ann. Inst. Fourier},
  volume  = {46},
  year    = {1996},
  pages   = {73--88}
}

@article{Hashimoto,
    author = {Hashimoto, Kenji},
    title = {Finite symplectic actions on the $K3$ lattice},
    journal = {Nagoya Math. J.},
    volume = {206},
    year = {2012},
    pages={99--153}
}

@article{DW,
    author = {Donten-Bury, M. and  Wi{ś}niewski, J. A.},
    title = {On 81 symplectic resolutions of a 4-dimensional quotient by a group of order 32},
    journal = {Kyoto J. Math.},
    volume = {57},
    number = {2},
    year = {2017},
    pages = {395–434}
}

@misc {Singular,
label = {Sing},
 title = {{\sc Singular} {4-4-0} --- {A} computer algebra system for polynomial computations},
 author = {Decker, Wolfram and Greuel, Gert-Martin and Pfister, Gerhard and Sch\"onemann, Hans},
 year = {2024},
 howpublished = {\url{http://www.singular.uni-kl.de}},
}

@manual{Sage,
label={Sage},
  Key          = {SageMath},
  Author       = {{The Sage Developers}},
  Title        = {{S}ageMath, the {S}age {M}athematics {S}oftware {S}ystem ({V}ersion 10.3)},
  note         = {{\tt https://www.sagemath.org}},
  Year         = {2024},
}

@misc{ATLAS,
    label = {ATLAS},
    title = {{ATLAS of Finite Group Representations - Version 3} {ATLAS: Alternating group $A_7$}},
    author = {Robert Wilson and Peter Walsh and Jonathan Tripp and Ibrahim Suleiman and Richard Parker and Simon Norton and Simon Nickerson and Steve Linton and John Bray and Rachel Abbott},
    howpublished = {\url{http://brauer.maths.qmul.ac.uk/Atlas/v3/alt/A7/}},
}

@manual{GAP,
    label={GAP},
    organization = {The GAP~Group},
    title= {GAP -- Groups, Algorithms, and Programming, Version 4.15.1},
    year= {2025},
    note         = {{\tt https://www.gap-system.org}}
}

@Misc{M2,
    label = {M2},
    author = {Grayson, Daniel R. and Stillman, Michael E.},
    title = {Macaulay2, a software system for research in algebraic geometry},
    howpublished = {Available at \url{http://www2.macaulay2.com}}
}

@article{BCHM2010,
  author  = {Birkar, Caucher and Cascini, Paolo and Hacon, Christopher D. and McKernan, James},
  title   = {Existence of minimal models for varieties of log general type},
  journal = {J. Amer. Math. Soc.},
  year    = {2010},
  volume  = {23},
  number  = {2},
  pages   = {405--468},
}

@article{Yam, 
    author={Yamagishi, Ryo},
    title={On smoothness of minimal models of quotient singularities by finite subgroups of {$SL_n(\mathbb{C})$}},
    volume={60}, 
    number={3}, 
    journal={Glasg. Math. J.},
    year={2018}, 
    pages={603–634}
}

@article{WW,
  title={Small contractions of symplectic 4-folds},
  author={Wierzba, J. and Wi{ś}niewski, J. A.},
  journal={Duke Math. J.},
  year={2002},
  volume={120},
  pages={65-95},
}

@article{PR,
    author = {Perego, A. and Rapagnetta, A.},
    title = {Irreducible symplectic varieties from moduli spaces of sheaves on {K3} and {A}belian surfaces},
    journal = {Algebr. Geom.},
    volume = {10},
    number ={3},
    year = {2023},
    pages={348–393}
}

@article{MT,
    author = {Markushevich, D.G. and Tikhomirov, A. S.},
    title = {New symplectic {V}-manifolds of dimension four via the relative compactified {P}rymian},
    journal = {Int. J. Math.},
    volume = {18},
    number ={10},
    year = {2007},
    pages={1187-1224}
}

@article{ASF,
    author = {Arbarello, E. and Saccà, G. and Ferretti, A.},
    title = {Relative Prym varieties associated to the double cover of an Enriques surface},
    journal = {J. Differential Geom.},
    volume = {100},
    number = {2},
    year = {2015},
    pages = {191--250}
}

@article{Prym,
title = {Irreducible symplectic varieties via relative {P}rym varieties},
journal = {Adv. Math.},
volume = {490},
pages = {110826},
year = {2026},
author = {Emma Brakkee and Chiara Camere and Annalisa Grossi and Laura Pertusi and Giulia Saccà and Sasha Viktorova}
}

@article{MenetCyclic,
    author = {Menet, G.},
    title = {On the integral cohomology of quotients of manifolds by cyclic groups},
    journal = {J. Math. Pures Appl.},
    volume = {119},
    number ={9},
    year = {2018},
    pages ={280–325}
}

@article{KapMen,
    author = {Kapfer, S. and Menet, G.},
    title = {Integral cohomology of the generalized {K}ummer fourfold},
    journal = {Algebr. Geom.},
    volume ={5},
    number ={5},
    year = {2018},
    pages={523-567}
}

@article{Menet2,
title = {Beauville–Bogomolov lattice for a singular symplectic variety of dimension 4},
journal = {J. Pure Appl. Algebra},
volume = {219},
number = {5},
pages = {1455-1495},
year = {2015},
author = {Grégoire Menet}
}

@unpublished{Mazzon,
    author = {Enrica Mazzon},
    title = {Terminalizations of quotients of {F}ano varieties of lines on cubic fourfolds},
    note = {(Preprint) arXiv:2602.16492}
}

@article{Verbitsky,
  author  = {Verbitsky, Misha},
  title   = {Holomorphic symplectic geometry and orbifold singularities},
  journal = {Asian J. Math.},
  year    = {2000},
  volume  = {4},
  number  = {3},
  pages   = {553--564},
}

@article{Brieskorn,
    author = {Brieskorn, Egbert},
    title = {Rationale {S}ingularit{\"a}ten komplexer {F}l{\"a}chen},
    journal = {Invent. Math.},
    volume ={4},
    year = {1968},
    pages ={1432-1297}
}

@article{marisia,
    author = {Maria Donten-Bury},
    title = {Cox rings of minimal resolutions of surface quotient singularities},
    journal = {Glasgow Math. J.},
    volume ={58},
    year = {2016},
    pages ={325–355}
}

@book{ADHL,
    author = {Arzhantsev, I. and Derenthal, U. and Hausen, J. and Laface, A.},
    title = {Cox Rings},
    publisher = {Cambridge University Press, New York},
    year = {2014}
}

@article{BerchtoldHausen,
author = {Berchtold, F. and Hausen, J.},
year = {2006},
month = {11},
pages = {483-516},
title = {{GIT}-{E}quivalence beyond the ample cone},
volume = {54},
journal = {Michigan Math. J.},
}

@article{SchmittClassGroup,
author = {Schmitt, Johannes},
title = {The class group of a minimal model of a quotient singularity},
journal = {Bull. Lond. Math. Soc.},
volume = {56},
number = {9},
pages = {2777-2793},
year = {2024}
}

@article{Luna,
    author = {Luna, Domingo},
    title = {Slices {é}tales},
    journal ={Bull. Soc. Math. France},
    volume={33},
    year = {1973},
    pages = {81-105}
}

\end{document}